\documentclass[11pt,twoside]{amsart}
\usepackage{amsfonts}
\usepackage{amsmath}
\usepackage{amssymb}
\usepackage[latin1]{inputenc} % pour les accons
\usepackage{mathrsfs}
\newtheorem{thm}{Theorem}
\newtheorem{pro}{Proposition}
\newtheorem{cor}{Corollary}
\newtheorem{lem}{Lemma}
\theoremstyle{definition}
\newtheorem{defn}{Definition}

\theoremstyle{remark}
\newtheorem{rem}{Remark}

\newtheorem{exe}{Example}

\newenvironment{proofof}{{\noindent {\it Proof of }}}{\hfill \finpr \\ }
\numberwithin{equation}{section}

\def\Supp{\mathop{\rm Supp}}
\newcommand{\cal}       {\mathcal}
\def\finpr{\hfill \hbox{
\vrule height 1.453ex  width 0.093ex  depth 0ex \vrule height
1.5ex  width 1.3ex  depth -1.407ex\kern-0.1ex \vrule height
1.453ex  width 0.093ex  depth 0ex\kern-1.35ex \vrule height
0.093ex  width 1.3ex  depth 0ex}}
\def\cb{{\Bbb C}}
\def\rb{{\Bbb R}}

          \def\L{\Lambda}         \def\l{\lambda}
\def\nb{\mathbb N}       \def\ds{\displaystyle}  
           \let\L=\longrightarrow  
                  \let\l=\rightarrow
               
                 \def\ds{\displaystyle}

\def\1{1\!\rm l}

\title[ ]{Lelong-Jensen formula, Demailly-Lelong numbers and weighted degree of positive supercurrents}
\author[ ]{Fredj Elkhadhra and Khalil Zahmoul}
\address{Department of Mathematics\\ Higher School of Sciences and Technology of Hammam Sousse\\ University of Sousse\\ Tunisia.}
\email{fredj.elkhadhra@essths.u-sousse.tn, zahmoul.khalil@essths.u-sousse.tn}
\begin{document}
\begin{abstract}  The goal of this work is to extend the concepts of  generalized Lelong number of positive currents investigated by Skoda, Demailly and Ghiloufi in complex analysis, to weakly positive supercurrents on the real superspaces. We generalize then a result of Lagerberg when the supercurrent is closed as well as a very recent result of Berndtsson for minimal supercurrents associated to submanifolds of $\rb^n$. The main tool is a variant of the well-known Lelong-Jensen formula in the superformalism case. Moreover, we extend to our setting various interesting theorems  in complex analysis such as Demailly and Rashkovskii comparison theorems. We also complete the work begun by Lagerberg on the degree of  positive closed supercurrents and we prove a removable singularities result for positive supercurrents.
\end{abstract} \maketitle
%\tableofcontents
%\newpage
\section{Introduction}
In complex analysis, Lelong numbers of positive closed currents, as generalized by Demailly, have interesting applications in many domains such as complex analytic and algebraic theory, and number theory. Roughly speaking, Lelong numbers can be seen as a generalization of multiplicity of analytic set at a singular point, to positive closed currents. This concepts has been extended by many authors for an important class of currents by replacing the closedness property by a plurisubharmonicity one. The main tool for the existence of Lelong numbers is Lelong-Jensen formula which also becomes useful in studying the growth at infinity of positive currents. In \cite{5}, Lagerberg has introduced a notion of positive closed supercurrents on finite dimensional real vector spaces. By a strong connection with the complex setting, he succeded in proving the existence of Lelong numbers and he studied some others notions such as the direct image and the degree of  positive supercurrents. The important fact in this work appears when investigating the definition of the associated Monge-Amp\`ere operator for convex functions. Indeed, there is a good link between the class of strongly positive closed supercurrents and tropical geometry. Recently, Berndtsson \cite{11} obtained many interesting results in the superformalism setting by establishing a relation between weakly positive supercurrents and minimal submanifolds of $\rb^n$. Furthermore, he gave variants of
some well-known results in complex analysis concerning the class of currents, namely the famous theorem of El-Mir on the extension of positive closed current across complete pluripolar sets. Very recently, by using the concepts of $m$-positivity in the complex Hessian theory, \c{S}ahin \cite{13}  has introduced the notions of $m$-positivity in this superformalism.
 In this paper, we begin with a refinement on  the $m$-pluripotential study given in \cite{13} by getting a connection with the real Hessian theory investigated by Trudinger and Wang \cite{15}. Next, we introduce the class of convex weakly positive supercurrents in a similar way as the class of plurisubharmonic positive currents in complex analysis and we are mostly concerned with the behaviour of such supercurrents in a neighborhood of a point or at the infinity. After proving the existence of the Lelong number of weakly positive supercurrents in several cases, we prove many related properties. Namely, Demailly comparison theorem in the local situation and the comparison Rashkovskii theorem at the infinity. Moreover, we establish some effective bounds for the masses of the supercurrents and for the generalized degree with respect to convex weights. Besides the introduction, the paper has five sections. In Section 2 we give the necessary notations and preliminaries on the superformalism theory from Lagerberg \cite{5}. Section 3, is reserved for a discussion on the concepts of $m$-positivity and $m$-convexity as presented by \c{S}ahin \cite{13}. We also deal with the definition and the continuity of the corresponding $m$-superHessian operator. In Section 4, we will present a superformalism version of the well-known Lelong-Jensen formula in the complex setting. As an application, we prove the existence of Lelong numbers of weakly positive supercurrents in various cases. In Section 5, we prove a superformalism counterparts of the weighted degree of positive currents and the comparison theorems of Demailly \cite{3} and Rashkovskii \cite{6} in the complex theory. Finally, in Section 6 we investigate the extension of positive supercurrents in a strong way with the lines of Dabbek, Elkhadhra and El Mir \cite{31} in the complex setting and which generalize the result given by Berndtsson \cite{11} for minimal supercurrents.
 \section{Preliminaries}
 This part is a background on the superforms and supercurrents concepts introduced by \cite{5} that will be used throughout this paper. We adopt definitions and notations from \cite{5} and \cite{11}. Assume that $V$ and $W$ are two $n$-dimensional vector spaces over $\rb$, so that $x=(x_{1},...,x_{n})$ and $\xi=(\xi_{1},...,\xi_{n})$ are the corresponding coordinates. Let $J:V \rightarrow W$ be an isomorphism such that $J(x)=\xi$, and denote its inverse by $J$ as well, so that $J(\xi)=-x$, if $x\in V$ is the element for which $J(x)=\xi$. Setting $E=V \times W=\lbrace(x,\xi);\,x\in V,\,\xi\in W\rbrace$, and observe that the map $J$ extends over $E$ by means of $J(x,\xi)=(J(\xi),J(x))$, so that $J^2=-id$. Let $p,q$ be two integers such that $0\leqslant p,q\leqslant n$. A smooth superform on $E$ of bidegree $(p,q)$ is a form
$$\alpha=\sum_{K,L}\alpha_{KL}(x)dx_{K}\wedge d\xi_{L},$$
where $K=(k_{1},...,k_{p})$, $dx_{K}=dx_{k_{1}}\wedge ...\wedge dx_{k_{p}}$, $L=(l_{1},...,l_{q})$, $d\xi_{L}=d\xi_{l_{1}}\wedge ...\wedge d\xi_{l_{q}}$ and each map $(x,\xi)\mapsto \alpha_{KL}(x)$ is smooth and depends only on $x$. In particular, if $p=q$ we say that $\alpha$ is symmetric if and only if $\alpha_{KL}=\alpha_{LK}$  $\forall K,L$. In the remainder of this paper, we denote by ${\mathscr E}^{p,q}:={\mathscr E}^{p,q}(E)$, the set of smooth superforms on $E$ of bidegree $(p,q)$. It is clear that $J^{\ast}(dx_{i})=d\xi_{i}$ and ${J}^{\ast}(d\xi_{i})=-dx_{i}$. In order to simplify the notation, denote by $J$ the operator $J^\ast$, which can be extended on ${\mathscr E}^{p,q}$  as a map denoted again by $J :{\mathscr E}^{p,q}\rightarrow{\mathscr E}^{q,p}$, and defined by
$$J(\alpha)=J\left(\sum_{K,L} \alpha_{KL}(x)dx_{K}\wedge d\xi_{L}\right)=(-1)^q\sum_{K,L} \alpha_{KL}(x)d\xi_{K}\wedge dx_{L},\quad \forall\alpha\in{\mathscr E}^{p,q}.$$
It is clear that $J$ is an almost complex structure. Moreover, if $\alpha\in{\mathscr E}^{p,p}$, then, $\alpha$ is symmetric if and only if  $J(\alpha)=\alpha$. Similarly as in the complex setting, we introduce three notions of positivity on ${\mathscr E}^{p,p}$. Let us consider the K\"ahler form in $\rb^n\times\rb^n$ to be $\beta:=\frac{1}{2}dd^{\#}|x|^2=\sum_{i=1}^{n}dx_{i}\wedge d\xi_{i}\in{\mathscr E}^{1,1}$. It is not hard to see that $\beta^{n}=n!dx_{1}\wedge d\xi_{1}\wedge...\wedge dx_{n}\wedge d\xi_{n}$. According to \cite{5}, a superform $\varphi\in{\mathscr E}^{n,n}$ is said to be positive ($\varphi\geqslant 0$) if $\varphi=g\beta^{n}$, where $g$ is a positive function. Let $\varphi\in{\mathscr E}^{p,p}$ be  symmetric. We say that $\varphi$ is :
  \begin{enumerate}
 \rm\item  weakly positive if $\varphi\wedge \alpha_{1}\wedge J(\alpha_{1})\wedge ...\wedge \alpha_{n-p}\wedge J(\alpha_{n-p})\geqslant 0,\ \forall\alpha_{1},...,\alpha_{n-p}\in{\mathscr E}^{1,0}.$
 \rm\item  positive if $\varphi\wedge \sigma_{n-p}\alpha\wedge J(\alpha)\geqslant 0,\ \forall\alpha\in{\mathscr E}^{n-p,0};\ \sigma_{k}=(-1)^{\frac{k(k-1)}{2}}.$
 \rm\item  strongly positive if $\varphi=\sum_{s=1}^{N}\lambda_{s}\alpha_{1,s}\wedge J(\alpha_{1,s})\wedge...\wedge \alpha_{p,s}\wedge J(\alpha_{p,s});\ \lambda_{s}\geqslant 0,\  \alpha_{i,s}\in{\mathscr E}^{1,0}.$
  \end{enumerate}
 Assume that $\alpha\in{\mathscr E}^{n,n}$, then there exists a function $\alpha_0$ defined on $V$ such that  $$\alpha=\alpha_{0}dx_{1}\wedge ...\wedge dx_{n}\wedge d\xi_{1}\wedge ...\wedge d\xi_{n}.$$  According to \cite{5}, when an orientation on $V$ is chosen and $\alpha_0$ is integrable, the integral of $\alpha$ is defined by setting  $$\int_{E}\alpha=\int_{V}\alpha_{0}dx_{1}\wedge ...\wedge dx_{n}.$$
The operators $d$ and $d^{\#}$ are of type $(1,0)$ and $(0,1)$ respectively and acting on ${\mathscr E}^{p,q}$ by the following expressions $d=\sum_{i=1}^{n}\partial_{x_{i}}\wedge dx_{i}$ and $d^{\#}=\sum_{j=1}^{n}\partial_{x_{j}}\wedge d\xi_{j}$. Similarly with the complex setting, we see easily that  $d^{2}=(d^{\#})^{2}=0$ and $dd^{\#}=-d^{\#}d$. Moreover, in this situation we can present a Stokes formula as follow: Assume that $\Omega\subset V$ is a smooth open bounded subset and let $\alpha\in{\mathscr E}^{n-1,n}$. Then, $$\int_{\Omega\times W}d\alpha=\int_{\partial\Omega\times W}\alpha.$$
Denote by ${\mathscr D}^{p,q}:={\mathscr D}^{p,q}(E)=\lbrace\alpha\in{\mathscr E}^{p,q}$; $\alpha$ is compactly supported in E$\rbrace$ whose topology can be defined by means of the inductive limit. We introduce the space of supercurrents of  bidegree $(p,q)$ as the topological dual of ${\mathscr D}^{n-p,n-q}$, noted ${\mathscr D}_{p,q}$. This means that a supercurrent $T$ of bidegree $(p,q)$ is nothing but a continuous linear form on ${\mathscr D}^{n-p,n-q}$. More precisely, $T$ is a superform of bidegree $(p,q)$ which has distributions coefficients depending only on $x$. That is   $$T=\sum_{| I|=p,| J|=q}T_{IJ}dx_{I}\wedge d\xi_{J},$$ where $T_{IJ}$ are distributions defined uniquely. In particular, as with superforms if $p=q$ we say that the supercurrent $T$ is symmetric if and only if $T_{IJ}=T_{JI}$  $\forall I,J$.  For any $\alpha\in {\mathscr D}^{n-p,n-q}$ denote by $\langle T,\alpha\rangle$ the action of $T$ on $\alpha$. A supercurrent $T$ is said to be closed if $dT=0$ and is $d^{\#}$-closed if $d^{\#}T=0$. It is not hard to see that a symmetric supercurrent $T$ is closed if and only if  $T$ is $d^{\#}$-closed. Assume that $\rho$ is a smooth radial function which is supported in the unit ball and such that $\int\rho(x)dx=1$. For $\varepsilon>0$, let $\rho_{\varepsilon}(x)=\frac{1}{\varepsilon^{n}}\rho(\frac{x}{\varepsilon})$. Hence, the regularization of $T$ is defined by
$$T\ast\rho_{\varepsilon}=\sum_{| I|=p,| J|=q}(T_{IJ}\ast\rho_{\varepsilon})dx_{I}\wedge d\xi_{J},$$
it is clear that the family $\lbrace T\ast\rho_{\varepsilon}\rbrace_{\varepsilon}\subset{\mathscr E}^{p,q}$ is weakly convergent  to $T$ when $\varepsilon\rightarrow 0$. Assume that $T$ is symmetric and of bidegree $(p,p)$, in analogy to the concepts of positivity in the complex context, $T$ is said to be :
\begin{enumerate}
\item weakly positive if $\langle T,\alpha\rangle\geqslant 0$ for any $\alpha\in {\mathscr D}^{n-p,n-p}$ strongly positive.
\item positive if $\langle T,\sigma_{n-p}\alpha\wedge J(\alpha)\rangle\geqslant 0$ for any $\alpha\in {\mathscr D}^{n-p,0}$.
\item strongly positive if  $\langle T,\alpha\rangle\geqslant 0$ for any $\alpha\in {\mathscr D}^{n-p,n-p}$ weakly positive.
\end{enumerate}
For $K\Subset\rb^n$ and $T$ is a supercurrent of order zero, we define the mass measure of $T$ on $K$ by $\|T\|_K=\sum_{IJ}|T_{IJ}|(K)$, where $T_{IJ}$ are the coefficients of $T$. Based on Proposition $4.1$ in \cite{5}, the mass of a positive supercurrent $T$ of bidimension $(p,p)$ on $K$ is proportional to the positive measure $T\wedge\beta^p(K)$. According to \cite{5}, we have the following  $dd^{\#}$-Lemma : Assume that $T\in{\mathscr D}_{1,1}$ is weakly positive and closed then there exists a convex function $f:V\longrightarrow\mathbb{R}$ such that $T=dd^{\#}f$.
For the sake of simplicity, in the rest of this paper we consider two copies of $\rb^n$, i.e. $V=W=\rb^n$ and we say form instead of superform and current instead of supercurrent.
\section{$m$-positivity and $m$-superHessian operator}
\subsection{$m$-positivity}  Building on the work of Douib and Elkhadhra \cite{16} on the Hessian complex pluripotential theory, \c{S}ahin \cite{13} has recently introduced the following notions of $m$-positivity in the superformalism context :
\begin{enumerate}
\item A symmetric form $\alpha$ of bidegree $(1,1)$ is said $m$-positive if at every point we have
$\alpha^j\wedge\beta^{n-j}\geqslant 0,\ \forall j=1,...,m$.
\item A symmetric current $T$ of bidegree $(p,p)$; $p\leqslant m\leqslant n$, is called $m$-positive if we have $T\wedge\beta^{n-m}\wedge\alpha_1\wedge...\wedge\alpha_{m-p}\geqslant 0,$ for all $m$-positive forms $\alpha_1,...,\alpha_{m-p}$ of bidegree $(1,1)$.
\item A function $u: V\L\rb\cup\{-\infty\}$ is called $m$-convex if it is subharmonic and the current $dd^{\#}u$ is $m$-positive. Denote by ${\mathscr C}_m$ the set of $m$-convex functions.
\end{enumerate}
As remarked by \c{S}ahin \cite{13}, if $p$ is an integer such that $1\leqslant p\leqslant m$ and $\alpha_1,...,\alpha_p$ are $m$-positive forms of bidegree $(1,1)$ then $\beta^{n-m}\wedge\alpha_1\wedge...\wedge\alpha_p$ is positive. Moreover, observe that this notion of m-positivity coincides with the one given by \cite{5} in the border case $m=n$. This is not the
case if $m<n$. Indeed, a simple computation yields that in $\rb^3\times\rb^3$, the form $\alpha=dx_1\wedge d\xi_1 + dx_2\wedge d\xi_2-\frac{1}{2} dx_3\wedge d\xi_3$ of bidegree $(1,1)$ is $2$-positive but not weakly positive. On the other hand, it is clear that every strongly positive current is automatically  $m$-positive.
 Now let us recall some basic facts about $m$-convex functions due to \cite{13} :
\begin{pro}\label{1}\
\begin{enumerate}
\item If $u$ is of class ${\mathscr C}^2$ then $u$ is $m$-convex if and only if $dd^{\#}u$ is $m$-positive form.
\item convex functions=${\mathscr C}_n\subset{\mathscr C}_{n-1}\subset\cdots\subset{\mathscr C}_1$= subharmonic functions.
\item If $u$ is $m$-convex then the standard regularization $u_j=u\ast\chi_j$ is smooth and $m$-convex. Moreover, $(u_j)_j$ decreases pointwise to $u$.
\item Let $u,v\in{\mathscr C}_m$ then $\max(u,v)\in{\mathscr C}_m$.
\item If $(u_\alpha)_\alpha\subset{\mathscr C}_m$, $u=\sup_\alpha u_\alpha<+\infty$ and $u$ is upper semicontinuous then $u$ is $m$-convex.
\end{enumerate}
\end{pro}
As an immediate consequence of the first statement, we see that if $u$ is of class ${\mathscr C}^2$ then $u$ is $m$-convex if and only if, $(dd^{\#}u)^k\wedge \beta^{n-k}$ is positive for $k=1,\cdots , m$.  However, it is not difficult to show that
\begin{equation}
(dd^{\#}u)^k\wedge\beta^{n-k}=\frac{(n-k)!}{n!} F_k[u]\beta^n,
\end{equation}
where $F_k[u]=\left[D^2u\right]_k$, is the well-known {\it $k$-Hessian operator} which was studied extensively by Trundinger and Wang \cite{15} and $\left[D^2u\right]_k$ denotes the sum of its $k\times k$ principal minors of the Hessian matrix of $u$. Consequently, $u$  is $m$-convex is equivalent to saying that  $F_k[u]\geqslant 0$ for $k=1,\cdots, m$. This coincides with the definition of $u$ to be $m$-convex in the sense of  Trundinger and Wang \cite{15}. It is a clear raison why we use the terminology of $m$-convex instead of $m$-subharmonic used by \c{S}ahin \cite{13}. In what follows, we give an important example of a well-known $m$-convex function which is fundamental in the real Hessian theory (see \cite{15} and \cite{17}) and will be used later in Theorem \ref{t6}.
\begin{exe}\label{x1} Setting $\varphi_m(x)=-\frac{1}{\left(\frac{n}{m}-2\right)|x|^{\frac{n}{m}-2}}$ if $m\neq\frac{n}{2}$ and $\log|x|$ otherwise, $\gamma=\frac{1}{2}d|x|^2$ and $\gamma^{\#}=\frac{1}{2}d^{\#}|x|^2$. Then, for $x\not=0$ and $m\neq\frac{n}{2}$, we have
$$\begin{array}{lcl}
dd^{\#}\varphi_m(x)=\ds\sum_{i,j=1}^{n}\frac{\partial^2\varphi_m(x)}{\partial x_i\partial x_j}dx_i\wedge d\xi_j&=&\ds \sum_{i,j=1}^{n}\frac{\partial (x_j|x|^{-\frac{n}{m}})}{\partial x_i}dx_i\wedge d\xi_j\\&=&\ds \left[\sum_{i=1}^{n}|x|^{-\frac{n}{m}}dx_i\wedge d\xi_i-\frac{n}{m}\sum_{i,j=1}^{n}x_ix_j|x|^{-\frac{n}{m}-2}dx_i\wedge d\xi_j\right]\\&=&\ds|x|^{-\frac{n}{m}}\left(\beta-\frac{n}{m}|x|^{-2}\gamma\wedge\gamma^{\#}\right).
\end{array}$$
It is clear that $(\gamma\wedge\gamma^{\#})^2=0$, therefore, for $s=1,...,m$, we have
$$\begin{array}{lcl}
(dd^{\#}\varphi_m(x))^s\wedge\beta^{n-s}&=&\ds|x|^{-\frac{ns}{m}}\left[\beta^{s}-\frac{ns}{m}|x|^{-2}\beta^{s-1}\wedge\gamma\wedge\gamma^{\#}\right]\wedge\beta^{n-s}
\\ &=&\ds|x|^{-\frac{ns}{m}}\left[\beta^n-\frac{ns}{m}|x|^{-2}\gamma\wedge\gamma^{\#}\wedge\beta^{n-1}\right]
\\&=&\ds|x|^{-\frac{ns}{m}}\left[\beta^n-\frac{s}{m}\beta^{n}\right]=\left(1-\frac{s}{m}\right)|x|^{-\frac{ns}{m}}\beta^{n}.
\end{array}$$
Now, for $x\neq0$ and $m=\frac{n}{2}$, a straightforward computation gives
$$dd^{\#}\varphi_m(x)=\ds\sum_{i,j=1}^{n}\frac{\partial^2(\log|x|)}{\partial x_i\partial x_j}dx_i\wedge d\xi_j=|x|^{-2}\left(\beta-2|x|^{-2}\gamma\wedge\gamma^{\#}\right).
$$
Thus, for $s=1,...,\dfrac{n}{2}$, we obtain
$$(dd^{\#}\varphi_m(x))^s\wedge\beta^{n-s}=\left(1-\frac{s}{m}\right)|x|^{-\frac{ns}{m}}\beta^n.$$
This leads to the conclusion that $\varphi_m$ is $m$-convex in both cases $m=\frac{n}{2}$ and $m\neq\frac{n}{2}$.
\end{exe}
\subsection{$m$-superHessian operator} Similarly as in the theory of complex Hessian operator, our purpose here is to define the wedge product $T\wedge\beta^{n-m}\wedge dd^{\#}u$, where $u$ and $T$ are not necessarily smooth. Let $T$ be a closed $m$-positive current of bidimension $(p,p)$; $m+p\geqslant n$ and  let $u$ be a locally bounded $m$-convex function. Since $T\wedge\beta^{n-m}$ is weakly positive and $u$ is locally bounded, then by \cite{5} the current $uT\wedge\beta^{n-m}$ has measure coefficients. Hence, we set $$T\wedge\beta^{n-m}\wedge dd^{\#}u=dd^{\#}(uT\wedge\beta^{n-m}).$$
Moreover, this current is weakly positive and closed. Indeed, the result is clear when $u$ is smooth. Otherwise, we consider a family of smooth regularized kernels $(\rho_{\varepsilon})_{\varepsilon>0}$. Therefore, $u_\varepsilon=u*\rho_{\varepsilon}$ is smooth and $m$-convex and the sequence of currents $u_\varepsilon T\wedge\beta^{n-m}$ converges weakly to $uT\wedge\beta^{n-m}$. By the continuity of $dd^{\#}$, we deduce that  $dd^{\#}(u_\varepsilon T\wedge\beta^{n-m})$ converges to  $dd^{\#}(uT\wedge\beta^{n-m})$ as currents. So, the positivity of $T\wedge\beta^{n-m}\wedge dd^{\#}u$, is a consequence of the one of $T\wedge\beta^{n-m}\wedge dd^{\#}u_\varepsilon$. More generally, if we assume that  $u_1,...,u_q$ are $m$-convex locally bounded functions on $\rb^n$; $q\leqslant p+m-n$, we can define by induction the following weakly positive closed current of bidimension $(p+m-n-q,p+m-n-q)$ :
$$T\wedge\beta^{n-m}\wedge dd^{\#}u_1\wedge dd^{\#}u_2\wedge...\wedge dd^{\#}u_q=dd^{\#}(u_1T\wedge\beta^{n-m}\wedge dd^{\#}u_2\wedge...\wedge dd^{\#}u_q).$$
It should be noted here that when $m=n$, such definition justified in \cite{5} as the unique adherent point of a family $ (T\wedge dd^{\#}u_{1}^{j}\wedge...\wedge dd^{\#}u_{q}^{j})_j$ which is locally uniformly bounded in masses, where the functions $u_{k}^{j}$ are smooth and convex and converges locally uniformly to $u_{k}$. Recently, when $m<n$, the same inductively definition was presented by \c{S}ahin \cite{13} in the particular cases either the $m$-convex functions $u_j$ are continuous or the functions are locally bounded and $T$ is a tropical variety of co-dimension $n-p$. This means that $$T=V_{f_1}\wedge\cdots\wedge V_{f_{n-p}}=dd^{\#}f_1\wedge\cdots\wedge dd^{\#}f_{n-p},$$  where $f_j$ are tropical polynomials and $V_{f_j}$ are the corresponding tropical hypersurfaces. By using a techniques which goes back to Demailly in the complex theory, we obtain the following proposition which improves a result of \cite{5} in the particular cases $m=n,\ T_k=T$ and $u_j^k$ are smooth and convex as well as a very recent result of \cite{13}, when $T_k=T$ and $u_j^k$ is the usual regularization of $u_j$.
\begin{pro}\label{2} Assume that $u_{1}^{k},...,u_{q}^{k}$ are sequences of $m$-convex functions which converge locally uniformly respectively to continuous $m$-convex functions $u_{1},...,u_{q}$. Assume that  $T_{k},T$ are $m$-positive closed currents of bidimension $(p,p)$; $m+p>n$, such that $T_k\wedge\beta^{n-m}$  converges weakly to $T\wedge\beta^{n-m}$. Then, in the sense of currents, we have :
\begin{enumerate}
\item $u_{1}^{k}T_k\wedge\beta^{n-m}\wedge dd^{\#}u_{2}^{k}\wedge...\wedge dd^{\#}u_{q}^{k}\ \longrightarrow\ u_{1}T\wedge\beta^{n-m}\wedge dd^{\#}u_{2}\wedge...\wedge dd^{\#}u_{q}$.
\item $T_{k}\wedge\beta^{n-m}\wedge dd^{\#}u_{1}^{k}\wedge dd^{\#}u_{2}^{k}\wedge...\wedge dd^{\#}u_{q}^{k}\ \longrightarrow\ T\wedge\beta^{n-m}\wedge dd^{\#}u_{1}\wedge dd^{\#}u_{2}\wedge...\wedge dd^{\#}u_{q}$.
\end{enumerate}
\end{pro}
\begin{proof} Thanks to the weak continuity of $dd^{\#}$, it is clear that (2) is a direct consequence of (1), then it suffices to prove (1). We proceed by induction on $q$. If $q=1$, let $u_{k}$ be a sequence of $m$-convex functions which converges uniformly on each compact subset to a continuous $m$-convex function $u$. Firstly, we consider a smooth regularization $u_{\varepsilon}=u\ast\rho_{\varepsilon}$ of $u$, and for simplicity of the proof setting $R=T\wedge\beta^{n-m}$ and $R_k=T_k\wedge\beta^{n-m}$. Then we have :
 $$u_{k}R_{k}-uR=(u_{k}-u)R_{k}+(u-u_{\varepsilon})R_{k}+u_{\varepsilon}(R_{k}-R)+(u_\varepsilon-u)R,\quad\forall\varepsilon>0.$$
Since $R_{k}$ is weakly positive and converges weakly to the weakly positive current $R$, then by Proposition $4.1$ in \cite{5} the currents $R_k,R$ are locally uniformly bounded in masses. Hence, $$\|(u_{k}-u)R_{k}\|_{K}\leqslant\|u_{k}-u\|_{L^{\infty}(K)}\|R_{k}\|_{K},\qquad\forall K\Subset\rb^n.$$ It follows that
$(u_{k}-u)R_{k}$ converges to $0$ when $k\rightarrow+\infty$. The same argument gives that  $(u-u_{\varepsilon})R_{k}$ and $(u_{\varepsilon}-u)R$ converge to $0$ when $\varepsilon\rightarrow0$. Since $u_{\varepsilon}$ is smooth, we have $u_{\varepsilon}(R_{k}-R)$ converges to $0$ when $k\rightarrow+\infty$. Consequently, we have proved that $u_{k}R_k$ converges weakly to $uR$. Now assume that $q\geqslant1$ and suppose that the property (1) is satisfied for $q$, and we are going to prove it for $q+1$. Let $u_{q+1}^k$ be a sequence of $m$-convex functions which converges locally uniformly to a continuous $m$-convex function $u_{q+1}$. We have $u_{1}^{k}R_k\wedge dd^{\#}u_{2}^{k}\wedge...\wedge dd^{\#}u_{q}^{k}$ is a sequence of currents of bidimension $(p-q-n+m+1,p-q-n+m+1)$ which converges as currents to $u_{1}R\wedge dd^{\#}u_{2}\wedge...\wedge dd^{\#}u_{q}$. Then, by the weak continuity of $dd^{\#}$, $R_k\wedge dd^{\#}u_{1}^{k}\wedge...\wedge dd^{\#}u_{q}^{k}$ is a sequence of weakly positive closed currents of bidimension $(p-q+n-m,p-q+n-m)$ which converges weakly to $R\wedge dd^{\#}u_{1}\wedge...\wedge dd^{\#}u_{q}$. Hence, $u_{q+1}^{k}R_k\wedge dd^{\#}u_{1}^{k}\wedge...\wedge dd^{\#}u_{q}^{k}$  converges  as currents to $u_{q+1}R\wedge dd^{\#}u_{1}\wedge...\wedge dd^{\#}u_{q}$.\end{proof}
\begin{rem} Before closing this section we state the following comments :
\begin{enumerate}
\item Concerning the potential theoretic aspects in the superformalism setting, let us recall that to each $m$-positive closed current $T$ of bidegree $(p,p)$ on an open subset $\Omega\Subset\rb^n$, we can associate a capacity in a similar way to  the capacity defined recently by \c{S}ahin \cite{13} and the one investigated by Dhouib and Elkhadhra \cite{16} in the complex Hessian theory. More precisely, if $K\subset\Omega$ is compact, we define the $m$-capacity of $K$ relative to $T$ by :
$$cap_{m,T}(K):=\sup\left\{ \int_{K\times\rb^n} T\wedge\beta^{n-m}\wedge (dd^{\#} u)^{m-p},~ u\in {\mathscr C}_m(\Omega), ~0\leqslant u\leqslant 1\right\},$$
and for every subset $E\subset\Omega$, $cap_{m,T}(E)=\sup\left\{cap_{m,T}(K),\ K\ {\rm compact\ in}\  E\right\}$.
When $T$ is a tropical variety, we recover the capacity of \c{S}ahin \cite{13}. Also, $cap_{m,T}$ can be viewed as a counterpart of the capacity introduced by \cite{16} in the complex Hessian theory. By going back to the comment before Example \ref{x1}, especially for the trivial current $T=1$, we get the so-called $m$-Hessian capacity defined by Trudinger and Wang \cite{15}. Such capacity shares the same properties as the preceding capacities. Furthermore, by an adaptation of the study given by \cite{16} in the complex Hessian theory, we can prove the quasicontinuity of  each locally bounded $m$-convex function with respect to $cap_{m,T}$.  This crucial property leads to relaxing the continuity condition of the functions $u_j$, this means that Proposition \ref{2} is still holds when the functions are locally bounded and $T_k=T$ (see the proof of Theorem 4.1 in \cite{13}). We leave the reader to consider by himself this more general situation.
\item In light of the above discussion, it is clear that a current of the form $dd^{\#}u_1\wedge...\wedge dd^{\#}u_k$ is $m$-positive, for $u_1,...,u_k$ locally bounded $m$-convex functions and $k\leqslant m$. Hence, since ${\mathscr C}_m\subset{\mathscr C}_{m-1}$, $dd^{\#}u_1\wedge...\wedge dd^{\#}u_k$ is again $(m-1)$-positive when $k\leqslant m-1$. However, as shown by the example stated before Proposition \ref{1}, we easily see that in general there is no link between $m$-positive and $(m-1)$-positive currents.
\end{enumerate}
\end{rem}
\section{Lelong-Jensen formula and Demailly-Lelong numbers}
 Analogously with the complex theory of positive currents, our goal in this section is to prove the existence of Lelong numbers of weakly positive currents in the superformalism setting. To do this, we let ourselves be inspired by the complex setting. Indeed, we follow the method of Lelong in the closed case, which has been generalized by Demailly \cite{2} and Skoda \cite{7} for the plurisubharmonic case and recently by Benali and Ghiloufi \cite{14} in the complex Hessian theory.
\subsection{Lelong-Jensen formula.} By following the proofs of Demailly \cite{2} and Skoda \cite{7} we are going to prove the corresponding version of Lelong-Jensen formula in our situation. Assume that $\varphi$ is a positive ${\mathscr C}^2$ function on $\rb^n$. For all  real numbers $r>0$ and $r_2>r_1>0$, setting :
$$B(r)=\{ x\in \rb^n;\ \varphi(x)<r\},\qquad S(r)=\{x\in \rb^n;\ \varphi(x)=r\}$$
$${\rm and}\quad B(r_1,r_2)=\{ x\in \rb^n;\ r_1<\varphi(x)<r_2\}.$$
Denote also by : $$\alpha=dd^{\#}\varphi^\frac{1}{2}\quad{\rm on\,the\ open\ set}\,\ \{\varphi>0\}\qquad{\rm and}\qquad\omega=dd^{\#}\varphi.$$
A direct computation gives :
\begin{equation}\label{e1}
\alpha=\frac{\omega}{2\varphi^{\frac{1}{2}}}-\frac{d\varphi\wedge d^{\#}\varphi}{4\varphi^{\frac{3}{2}}},\qquad\alpha^p=\frac{\omega^p}{2^p\varphi^{\frac{p}{2}}}-p\frac{\omega^{p-1}\wedge d\varphi\wedge d^{\#}\varphi}{2^{p+1}\varphi^{\frac{p+2}{2}}}.
\end{equation}
With these notations, we prove :
\begin{pro}\label{3} Assume that $T$ is a current of bidimension $(p,p)$ on $\rb^n\times\rb^n$, such that $T$ and $dd^{\#}T$ are symmetrical and have measure coefficients. For every $r_2>r_1>0$, we have :
$$\begin{array}{lcl}
\ds{1\over {2^p{r_2}^{\frac{p}{2}}}}\ds\int_{B(r_2)\times\rb^n}T\wedge\omega^p&-&\ds{1\over {2^p{r_1}^{\frac{p}{2}}}}\ds\int_{B(r_1)\times\rb^n}T\wedge\omega^p=\ds\int_{ B(r_1,r_2)\times\rb^n}T\wedge\alpha^p\\&+&\ds\left(\ds{1\over {2^p{r_1}^{\frac{p}{2}}}}-\ds{1\over {2^p{r_2}^{\frac{p}{2}}}}\right)\ds\int_0^{r_1}dt\ds\int_{B(t)\times\rb^n}dd^{\#} T\wedge\omega^{p-1}\\&+&\ds\int_{r_1}^{r_2}\left(\ds{1\over {2^p{t}^{\frac{p}{2}}}}-\ds{1\over {2^p{r_2}^{\frac{p}{2}}}}\right)dt\int_{B(t)\times\rb^n}dd^{\#}T\wedge\omega^{p-1}.
 \end{array}$$
\end{pro}
 As an immediate consequence, we see that if $T$ is a closed weakly positive current, $\varphi=|x|^2$, $\mathbb{B}(r)=\{x\in\rb^n,|x|<r\}$ and $\mathbb{B}(r_1,r_2)=\{x\in\rb^n,r_1<|x|<r_2\}$, then we recover the following formula due to Lagerberg \cite{5} :
$$\ds{1\over {{r_2}^{p}}}\ds\int_{\mathbb{B}(r_2)\times\rb^n}T\wedge\beta^p-\ds{1\over {{r_1}^{p}}}\ds\int_{\mathbb{B}(r_1)\times\rb^n}T\wedge\beta^p=\ds\int_{ \mathbb{B}(r_1,r_2)\times\rb^n}T\wedge\alpha^p.$$
For the proof of Proposition \ref{3} we need the following Lemma :
\begin{lem}\label{l1} Assume that $\psi$ is a ${\mathscr C}^1$ function on $\rb^n$ and $\gamma=\sum_{j,k}\gamma_{jk}\check{dx_j}\wedge\check{d\xi_k}$ is a  symmetric form of bidegree $(n-1,n-1)$ on  $\rb^n\times\rb^n$, where $\check{dx_j}=dx_1\wedge...\wedge dx_{j-1}\wedge dx_{j+1}\wedge...\wedge dx_n$ and similarly for $\check{d\xi_k}$. Then, we have $d\psi\wedge d^{\#}\gamma=-d^{\#}\psi\wedge d\gamma.$
\end{lem}
\begin{proof} By going back to the definition of the operators $d$ and $d^{\#}$, it is no difficult to get :
 $$\begin{array}{lcl}
 d\psi\wedge d^{\#}\gamma&=&\ds\sum_{s,j,k,t}\partial_{x_s}\psi\ \partial_{x_t}\gamma_{jk}\ dx_s\wedge d\xi_t\wedge\check{dx_j}\wedge\check{d\xi_k}\\&=&\ds\sum_{j,k}\partial_{x_j}\psi\ \partial_{x_k}\gamma_{jk}\ dx_j\wedge d\xi_k\wedge\check{dx_j}\wedge\check{d\xi_k}\\&=&\ds\left(\sum_{j,k}(-1)^{n+k+j-1}\partial_{x_j}\psi\ \partial_{x_k}\gamma_{jk}\right)\ dx_1\wedge...\wedge dx_n\wedge d\xi_1\wedge...\wedge d\xi_n,
 \end{array}$$
 Similarly,
 $$\begin{array}{lcl}
 d^{\#}\psi\wedge d\gamma&=&\ds\sum_{s,j,k,t}\partial_{x_s}\psi\ \partial_{x_t}\gamma_{jk}\ d\xi_s\wedge dx_t\wedge\check{dx_j}\wedge\check{d\xi_k}\\&=&\ds\sum_{j,k}\partial_{x_k}\psi\ \partial_{x_j}\gamma_{jk}\ d\xi_k\wedge dx_j\wedge\check{dx_j}\wedge\check{d\xi_k}\\&=&\ds-\sum_{j,k}\partial_{x_k}\psi\ \partial_{x_j}\gamma_{jk}\ dx_j\wedge d\xi_k\wedge\check{dx_j}\wedge\check{d\xi_k}\\&=&\ds-\left(\sum_{j,k}(-1)^{n+k+j-1}\partial_{x_k}\psi\ \partial_{x_j}\gamma_{kj}\right)\ dx_1\wedge...\wedge dx_n\wedge d\xi_1\wedge...\wedge d\xi_n,
 \end{array}$$
Therefore, since $\gamma_{jk}=\gamma_{kj}$ we obtain $d\psi\wedge d^{\#}\gamma=-d^{\#}\psi\wedge d\gamma$.\end{proof}
\begin{proofof}{\it Proposition \ref{3}.} Assume firstly that $T$ is of class ${\mathscr C}^\infty$. Using Stokes formula, we have
$$\begin{array}{lcl}
\ds\int_{B(r_1,r_2)\times\rb^n}T\wedge\alpha^p&=&\ds\int_{B(r_1,r_2)\times\rb^n}d\left(T\wedge d^{\#}\varphi^{\frac{1}{2}}\wedge\alpha^{p-1}\right)+\ds\int_{B(r_1,r_2)\times\rb^n}d^{\#}\varphi^{\frac{1}{2}}\wedge dT\wedge \alpha^{p-1}\\&=&I+II.
\end{array}$$
Let $j_t:S(t)\hookrightarrow \rb^n$. Since $j_t^* d\varphi=0$ and by (\ref{e1}), we get
$$\displaystyle j_t^*\alpha=\frac{j_t^*\omega}{2t^{\frac{1}{2}}}\qquad{\rm and}\qquad j_t^*\alpha^p=\ds\frac{j_t^*\omega^p}{{2^pt^{\frac{p}{2}}}}.$$
 By applying Lemma \ref{l1} for $\psi=\varphi^{\frac{1}{2}},\gamma=T\wedge\alpha^{p-1}$, the Fubini's theorem and Stokes formula give
$$\begin{array}{lcl}
II=-\ds\int_{B(r_1,r_2)\times\rb^n}d\varphi^{\frac{1}{2}}\wedge d^{\#}T\wedge\alpha^{p-1}&=&-\ds\int_{r_1}^{r_2}\frac{dt}{2t^{\frac{1}{2}}}\int_{S(t)\times\rb^n}d^{\#}T\wedge\alpha^{p-1}\\&=&
-\ds\int_{r_1}^{r_2}\frac{dt}{2^pt^{\frac{p}{2}}}\int_{B(t)\times\rb^n}dd^{\#}T\wedge\omega^{p-1}.
\end{array}$$
On the other hand, by applying twice Stokes formula, we obtain
$$\begin{array}{lcl}
I&=&\ds\int_{S(r_2)\times\rb^n}T\wedge d^{\#}\varphi^{\frac{1}{2}}\wedge
\alpha^{p-1}-\int_{S(r_1)\times\rb^n}T\wedge d^{\#}\varphi^{\frac{1}{2}}\wedge
\alpha^{p-1}\\&=&\ds\frac{1}{2^p{r_2}^{\frac{p}{2}}}\ds\int_{S(r_2)\times\rb^n}T\wedge d^{\#}\varphi\wedge\omega^{p-1}-\frac{1}{2^p{r_1}^{\frac{p}{2}}}\int_{S(r_1)\times\rb^n}T\wedge d^{\#}\varphi\wedge\omega^{p-1}\\&=&\ds\frac{1}{2^p{r_2}^{\frac{p}{2}}}\ds\int_{B(r_2)\times\rb^n}T\wedge \omega^{p}-\frac{1}{2^p{r_1}^{\frac{p}{2}}}\int_{B(r_1)\times\rb^n}T\wedge \omega^{p}\\&+&\ds\frac{1}{2^p{r_2}^{\frac{p}{2}}}\ds\int_{B(r_2)\times\rb^n}dT\wedge d^{\#}\varphi\wedge\omega^{p-1}-\frac{1}{2^p{r_1}^{\frac{p}{2}}}\int_{B(r_1)\times\rb^n}dT\wedge d^{\#}\varphi\wedge\omega^{p-1}.
\end{array}$$
Once again Lemma \ref{l1} for $\psi=\varphi,\gamma=T\wedge\omega^{p-1}$, the Fubini's theorem and Stokes formula yield
$$\begin{array}{lcl}
\ds\frac{1}{2^ps^{\frac{p}{2}}}\ds\int_{B(s)\times\rb^n}dT\wedge d^{\#}\varphi\wedge\omega^{p-1}&=&\ds\frac{1}{2^ps^{\frac{p}{2}}}\ds\int_{B(s)\times\rb^n}d\varphi\wedge d^{\#}T\wedge\omega^{p-1}\\&=&\ds\frac{1}{2^ps^{\frac{p}{2}}}\ds\int_0^sdt\ds\int_{S(t)\times\rb^n}d^{\#}T\wedge\omega^{p-1}\\&=&\ds\frac{1}{2^ps^{\frac{p}{2}}}\ds\int_0^sdt\ds\int_{B(t)\times\rb^n}dd^{\#}T\wedge\omega^{p-1}.
\end{array}$$
Now, take $s=r_2$ and $s=r_1$ and replace what in the preceding equation, then split the integral from $0$ to $r_2$ into a sum of two integrals one from $0$ to $r_1$ and the other from $r_1$ to  $r_2$, we obtain the desired formula. Finally, suppose only that $T$ and $dd^{\#}T$ are of order zero and consider a family of smooth regularized kernels $(\rho_\varepsilon)_{\varepsilon>0}$. Then, $T_\varepsilon=T*\rho_\varepsilon$ is a smooth form which converges as currents to $T$. After rewriting the formula of Proposition \ref{3} for  $T*\rho_\varepsilon$, we denote by ${\1}_{B(r)}$ the characteristic function of $B(r)$. So, we have $$\lim_{\varepsilon\rightarrow0}\int_{B(r)\times\rb^n}(T*\rho_\varepsilon)\wedge\omega^p=\lim_{\varepsilon\rightarrow0}\int_{\rb^n\times\rb^n}T\wedge[\rho_\varepsilon*({\1}_{B(r)}\omega^p)]=\int_{B(r)\times\rb^n}T\wedge\omega^p,$$
because $\rho_\varepsilon*({\1}_{B(r)}\omega^p)$ converges pointwise to ${\1}_{B(r)}\omega^p$ for $r$ such that $S(r)$ is negligible with respect to the masses of the currents $T$ and $dd^\#T$. We use the same arguments for the integrals involving $dd^{\#}T$.\end{proofof}
\begin{defn} A current $T$ of bidimension $(p,p)$ on $\rb^n\times\rb^n$ is said to be convex if  $dd^{\#}T$ is a weakly positive current. We say that $T$ is concave if $-T$ is convex, i.e. $dd^{\#}T$ is weakly negative.
\end{defn}
\begin{exe} \ \begin{enumerate}
\item Every convex function $u$ defines a convex current of degree zero. More generally, if $T$ is a weakly positive closed current and $u$ is a convex function, then the current $uT$ is convex. Another interesting example of a weakly positive convex and concave current is the current $T\wedge\beta^{p-1}$ of bidimension $(1,1)$, where $T$ is the so-called {\it minimal supercurrent} (i.e. $T$ is weakly positive and  $T\wedge\beta^{p-1}$ is closed) which is introduced and studied very recently by Berndtsson \cite{11}.
\item Let $M$ be a smooth $p$-dimensional submanifold of $\rb^n$. Let us first assume that $M$ is locally defined by $n-p$ equations $\rho_j=0$, such that $d\rho_j$ are linearly independent on $M$. Following the terminology of Berndtsson \cite{11}, by replacing $\rho_j$ by $\sum a_{jk}\rho_k=:\rho'_j$, for a suitable matrix of functions $a_{jk}$ and assuming that $n_j:=d\rho'_j$ are orthonormal on $M$, the current of integration on $M$ can be defined by $[M]:=n_1\wedge...\wedge n_{n-p}\star dS_M,$ where $dS_M$ is the surface measure on $M$ and the Hodge star indicates that we think of it as a current of degree zero. Next, Berndtsson defined the current associated to $M$ as $$[M]_s:=n_1\wedge n^{\#}_1\wedge...\wedge n_{n-p}\wedge n^{\#}_{n-p}\star dS_M,$$ where $n^{\#}_j=d^{\#}\rho'_j,\ \forall 1\leqslant j\leqslant p$. It is clear that $[M]_s$ is a positive symmetric current. For the computation of $d[M]_s$, $d^{\#}[M]_s$ and $dd^{\#}[M]_s$ we will have use for the $(1, 1)$-forms $$F_j:=dn^{\#}_j,\quad F^{\#}_j:=d^{\#}n_j=-F_j.$$
For more information the reader can go back to \cite{11}, but by applying contraction as defined by Berndtsson, it is easy to get
$$d[M]_s=\sum_{j=1}^{n-p}F_j\wedge n^{\#}_j\rfloor[M]_s,\quad d^{\#}[M]_s=\sum_{j=1}^{n-p}F^{\#}_j\wedge n_j\rfloor[M]_s=-\sum_{j=1}^{n-p}F_j\wedge n_j\rfloor[M]_s.$$
Which leads with a simple calculation to
$$\begin{array}{lcl}
dd^{\#}[M]_s=\ds d\left(\sum_{j=1}^{n-p}F^{\#}_j\wedge n_j\rfloor[M]_s\right)&=&-\ds\sum_{j=1}^{n-p}F^{\#}_j\wedge n_j\rfloor d[M]_s\\&=&-\ds\sum_{j=1}^{n-p}F^{\#}_j\wedge n_j\rfloor\left(\sum_{k=1}^{n-p}F_k\wedge n^{\#}_k\rfloor[M]_s\right)\\&=&-\ds\sum_{j=1}^{n-p}F^{\#}_j\wedge\sum_{k=1}^{n-p}F_k\wedge(n_j\rfloor(n^{\#}_k\rfloor[M]_s))\\&=&\ds\sum_{j,k=1}^{n-p}F_j\wedge F_k\wedge(n_j\rfloor(n^{\#}_k\rfloor[M]_s)).
\end{array}$$
So, it is not hard to see that:
\begin{itemize}
\item[i.] For $p=n-1$, $dd^{\#}[M]_s=F_1\wedge F_1\wedge(n_1\rfloor(n^{\#}_1\rfloor[M]_s))$. Then, if we assume that the function $\rho_1$ is convex, the current $[M]_s$ is convex.
\item[ii.] For $1<p\leqslant n-1$, if we assume that $n^{\#}_j\wedge F_j\wedge n_k\wedge F_k$ is a strongly positive form for every $1\leqslant j,k\leqslant n-p$, then the current $[M]_s$ is convex.
\end{itemize}
\end{enumerate}
\end{exe}
As a consequence of the proof of Proposition \ref{3}, we obtain the following analogous formula due to Demailly \cite{2} in the complex theory :
\begin{cor}\label{c1} With the same hypothesis as in Proposition \ref{3} and from the last proof, for every $r_2>r_1>0$ we have :
$$\begin{array}{lcl}
\ds\int_{r_1}^{r_2}\frac{dt}{2^{p}t^{\frac{p}{2}}}\int_{B(t)\times\rb^n}dd^{\#}T\wedge\omega^{p-1}+\ds\int_{B(r_1,r_2)\times\rb^n}T\wedge\alpha^p&=&\ds\frac{1}{2^p{r_2}^{\frac{p}{2}}}\ds\int_{S(r_2)\times\rb^n}T\wedge d^{\#}\varphi\wedge
\omega^{p-1}\\&-&\ds\frac{1}{2^p{r_1}^{\frac{p}{2}}}\int_{S(r_1)\times\rb^n}T\wedge d^{\#}\varphi\wedge\omega^{p-1}.
\end{array}$$
Furthermore, if $\varphi^{\frac{1}{2}}$ is convex,  $T$ is weakly positive and $T\wedge\omega^{p-1}$ is convex, then the map
$$r\longmapsto \ds\frac{1}{2^p{r}^{\frac{p}{2}}}\ds\int_{S(r)\times\rb^n}T\wedge d^{\#}\varphi\wedge
\omega^{p-1},$$
is increases.
\end{cor}
{\bf Particular case}: For $\varphi=|x|^2$ and by (\ref{e1}), for $x\in\rb^n\smallsetminus\{0\}$, we have
$$\begin{array}{lcl}
\alpha^n&=&\ds\frac{\beta^n}{|x|^n}-n\frac{\beta^{n-1}\wedge d|x|^2\wedge d^{\#}|x|^2}{4|x|^{n+2}}\\&=&\left[\ds\sum_{j,k} (\delta_{jk}-nx_jx_k|x|^{-2})dx_j\wedge d\xi_k\right]\wedge|x|^{-n}\beta^{n-1}=0.\end{array}$$
And,
$$\begin{array}{lcl}
d^{\#}\varphi\wedge\beta^{n-1}&=&\ds\left(\sum_{i=1}^{n}2x_id\xi_i\right)\wedge\left((n-1)!\sum_{j=1}^{n}\widehat{dx_j\wedge d\xi_j}\right)\\&=&\ds2(n-1)!\sum_{i=1}^{n}x_i\widehat{dx_i}\\&=&\ds2(n-1)!\left(\sum_{i=1}^{n}(-1)^{i-1}x_i\check{dx_i}\right)\wedge d\xi_1\wedge...\wedge d\xi_n,
\end{array}$$
where $\widehat{dx_i\wedge d\xi_i}=dx_1\wedge d\xi_1\wedge...\wedge dx_{i-1}\wedge d\xi_{i-1}\wedge dx_{i+1}\wedge d\xi_{i+1}\wedge...\wedge dx_n\wedge d\xi_n$ and similarly for $\widehat{dx_i}$ and $\widehat{d\xi_i}$. Therefore, if $T=f$ is a positive function such that $\Delta f$ is a measure, then since
$$
dd^{\#}f\wedge\beta^{n-1}=\ds\frac{2}{n}\Delta f.\beta^n=\ds2(n-1)!\Delta f.dx_1\wedge...\wedge dx_n\wedge d\xi_1\wedge...\wedge d\xi_n,$$
the equality of Corollary \ref{c1} becomes :
$$\int_{r_1}^{r_2}\frac{dt}{t^{n}}\int_{B(t)}\Delta fd\lambda=\frac{1}{{r_2}^{n}}\int_{S(r_2)}fd\sigma-\frac{1}{{r_1}^{n}}\int_{S(r_1)}fd\sigma,$$
where $d\lambda=dx_1\wedge...\wedge dx_n$ and $d\sigma=\sum_{i=1}^{n}(-1)^{i-1}x_i\check{dx_i}$. In particular, when $\Delta f$ is positive, the map $r\mapsto \frac{1}{{r}^{n}}\int_{S(r)}fd\sigma$ is increases and convex in $\log r$. By considering open subsets of $\cb^n\equiv\rb^{2n}$, this fact was observed by Demailly \cite{2}.
\subsection{Demailly-Lelong numbers.}
\begin{defn}\label{.1} Let $\varphi$ be a function as in the previous section and $T$ be a current of bidimension $(p,p)$ on $\rb^n\times\rb^n$. We define the Lelong number of $T$ relative to the weight $\varphi$ by  $$\nu_T(\varphi)=\lim_{r\rightarrow0}\nu_T(\varphi,r)\quad{\rm(when\ it\ exists)},$$
where, $\nu_T(\varphi,r)=\ds\frac{1}{2^pr^\frac{p}{2}}\int_{B(r)\times\mathbb{R}^n}T\wedge\omega^{p}$.
\end{defn}
\begin{thm}\label{t1} Let $T$ be a weakly positive current of bidimension $(p,p)$ on $\rb^n\times\rb^n$ and $\varphi$ be a ${\mathscr C}^2$ positive function on $\rb^n$ such that $\varphi^\frac{1}{2}$ and $T\wedge\omega^{p-1}$ are convex. Then the map
$$r\longmapsto \ds{1\over {2^p{r}^{\frac{p}{2}}}}\ds\int_{B(r)\times\rb^n}T\wedge\omega^p,$$ is positive and increases. In particular, the Lelong number of $T$ relative to the weight $\varphi$ exists.
\end{thm}
Theorem \ref{t1} is the corresponding result of the one obtained by \cite{7} in the complex setting.
\begin{proof} Since $\varphi^{\frac{1}{2}}$ is convex, $\varphi$ is also convex. Both weak positivity of $T$ and the convexity of $T\wedge\omega^{p-1}$ implies that the measures $T\wedge\alpha^p$, $T\wedge\omega^p$ and $dd^{\#}T\wedge\omega^{p-1}$ are positive. According to Proposition \ref{3}, it is clear that the map $r\mapsto\nu_T(\varphi,r)$ is positive and increases.\end{proof}
\begin{exe} Let $\Omega$ be an open subset of $\rb^n$, $T$ be a weakly positive closed current of bidimension $(p,p)$ on $\Omega\times\rb^n$ and $f$ be a convex positive function on $\Omega$. By combining Theorem \ref{t1} and the fact that $f$ is continuous, it is not hard to prove that the Lelong number $\nu_{fT}$ exists on every point of $\Omega$, and we have
$$\nu_{fT}(a)=f(a)\nu_T(a),\qquad\forall a\in\Omega.$$
\end{exe}
\begin{cor}\label{c2} Assume that $\varphi=|x-a|^2$ and $\mathbb{B}(a,r)=\{x\in\rb^n,|x-a|<r\}$, $\forall a\in\rb^n$. Then, for every weakly positive current $T$ of  bidimension $(p,p)$ on $\rb^n\times\rb^n$ such that $T\wedge\beta^{p-1}$ is convex, the positive function $$\nu_T(a,.):r\longmapsto\frac{1}{r^p}\int_{\mathbb{B}(a,r)\times\mathbb{R}^n}T\wedge\beta^{p}$$ is increases with respect to $r$. In particular, the limit $$\nu_T(a):=\lim_{r\rightarrow0}\nu_T(a,r),$$ exists and will be called the Lelong number of $T$ at $a$.
\end{cor}
This result generalizes the existence of Lelong numbers in the case where $T$ is a weakly positive closed current proved by \cite{5}. Moreover, Berndtsson \cite{11} establish Corollary \ref{c2} in the particular case where $T$ is a minimal supercurrent. Corollary \ref{c2} is also a variant of the well-known result for positive plurisubharmonic currents (see Demailly \cite{2} and Skoda \cite{7}) in the complex setting.
\begin{cor} Let $\Omega$ be an open subset of $\rb^n$ and $T$ be a weakly negative convex current of bidimension $(p,p)$ on $\Omega\times\rb^n$. Then, for every $a\in\Omega$ and $0<r_0\leqslant d(a,\partial\Omega)$, there exists $c_0<0$ such that for any $0<r\leqslant r_0$, we have $$\nu_T(a,r)\geqslant r\nu_{dd^{\#}T}(a,r_0)+c_0.$$
\end{cor}
\begin{proof} Without loss of generality, we can assume that $a=0$. For $r\leqslant r_0$, we set : $$\Upsilon_T(r)=\nu_T(0,r)-r\nu_{dd^{\#}T}(0,r_0).$$
Thanks to Proposition \ref{3}, for any $r_1<r_2\leqslant r_0$, we have :
$$\begin{array}{lcl}
\Upsilon_T(r_2)-\Upsilon_T(r_1)&=&\nu_T(0,r_2)-\nu_T(0,r_1)-(r_2-r_1)\nu_{dd^{\#}T}(0,r_0)\\&=&\ds\int_{\mathbb{B}(r_1,r_2)\times\rb^n}T\wedge(dd^{\#}|x|)^p+\ds\int_0^{r_1}\left(\ds{1\over {{r_1}^{p}}}-\ds{1\over {{r_2}^{p}}}\right)t^p\nu_{dd^{\#}T}(0,t)dt\\&+&\ds\int_{r_1}^{r_2}\left(\ds{1\over {t^{p}}}-\ds{1\over {{r_2}^{p}}}\right)t^p\nu_{dd^{\#}T}(0,t)dt-(r_2-r_1)\nu_{dd^{\#}T}(0,r_0)\\&=&\ds\int_{\mathbb{B}(r_1,r_2)\times\rb^n}T\wedge(dd^{\#}|x|)^p-(r_2-r_1)\nu_{dd^{\#}T}(0,r_0)
\\&+&\ds\int_{r_1}^{r_2}\nu_{dd^{\#}T}(0,t)dt-\int_{0}^{r_2}\left(\frac{t}{r_2}\right)^p\nu_{dd^{\#}T}(0,t)dt+
\int_{0}^{r_1}\left(\frac{t}{r_1}\right)^p\nu_{dd^{\#}T}(0,t)dt
\\&=&\ds\int_{\mathbb{B}(r_1,r_2)\times\rb^n}T\wedge(dd^{\#}|x|)^p+\int_{r_1}^{r_2}\left(\nu_{dd^{\#}T}(0,t)-\nu_{dd^{\#}T}(0,r_0)\right)dt
\\&-&\ds\int_{0}^{r_2}\left(\frac{t}{r_2}\right)^p\nu_{dd^{\#}T}(0,t)dt+\int_{0}^{r_1}\left(\frac{t}{r_1}\right)^p\nu_{dd^{\#}T}(0,t)dt.
\end{array}$$
      Since $T$ is weakly negative then $T\wedge(dd^{\#}|x|)^p$ is a negative Borel measure on $\mathbb{B}(r_1,r_2)$, so $$\int_{\mathbb{B}(r_1,r_2)\times\rb^n}T\wedge(dd^{\#}|x|)^p\leqslant0.$$
Moreover, $dd^{\#}T$ is a weakly positive closed current. Then, by Corollary \ref{c2}, $\nu_{dd^{\#}T}(0,.)$ is an increasing function on $]0,r_0]$. Thus, $$\int_{r_1}^{r_2}\left(\nu_{dd^{\#}T}(0,t)-\nu_{dd^{\#}T}(0,r_0)\right)dt\leqslant0.$$
Furthermore, if we set the function $$f:r\longmapsto-\frac{1}{r^p}\int_{0}^{r}t^p\nu_{dd^{\#}T}(0,t)dt,$$
then $f$ is continuous function on $]0,r_0]$ and we have :
$$f'(r)=\frac{pr^{p-1}}{r^{2p}}\int_{0}^{r}t^p\nu_{dd^{\#}T}(0,t)dt-\nu_{dd^{\#}T}(0,r)\leqslant\frac{p}{p+1}\nu_{dd^{\#}T}(0,r)-\nu_{dd^{\#}T}(0,r)\leqslant0$$
for almost every $0<r<r_0$. Hence, it is easy to see that $\Upsilon_T$ is a decreasing function on $]0,r_0]$, thus $\Upsilon_T(r)\geqslant\Upsilon_T(r_0)$ for every $0<r\leqslant r_0$. We conclude that $$\nu_T(a,r)\geqslant\Upsilon_T(r_0)+r\nu_{dd^{\#}T}(a,r_0),\qquad\forall0<r\leqslant r_0,$$ and the result follows by choosing for example $c_0=\min(0,\Upsilon_T(r_0))$.
\end{proof}
 Next, we give a version of a result recently obtained by Benali and Ghiloufi \cite{14} in the complex Hessian theory, which can be viewed as a generalization of Corollary \ref{c2}.
\begin{thm}\label{t6} Let $\varphi$ and $\mathbb{B}(a,r)$ be as in Corollary \ref{c2}. Assume that $T$ is an $m$-positive current of bidimension $(p,p)$ such that $T\wedge\beta^{p-1}$ is convex and $m+p>n$. Then, the limit $$\nu_T^{m}(a):=\lim_{r\l 0}r^{\frac{-n}{m}(m-n+p)}\int_{\mathbb{B}(a,r)\times\mathbb{R}^n}T\wedge\beta^{p},$$ exists and will be called the $m$-Lelong number of $T$ at $a$.
\end{thm}
\begin{rem}\
\begin{enumerate}
\item As a special case when $T=dd^{\#}u$, for $u$ is $m$-convex function, we recover the definition given by \cite{17} (modulo a constant). Notice here that such a definition depends on $m$, otherwise, it requires an additional condition that $u$ must be not $(m+1)$-convex.
\item Assume that $T$ is closed, $m$-positive and $(m-1)$-positive at the same times. Then, we easily see that  the $(m-1)$-Lelong number of $T$ vanishes. In particular, if $T$ is a strongly positive closed current, then the $j$-Lelong number of $T$ vanishes, for any $j\in\{p,...,n-1\}$. Indeed, $T$ is $m$-positive for any $m$ such that ${m+p>n}$.
\end{enumerate}
\end{rem}
\begin{proof}  Again, here the tool is a Lelong-Jensen formula and without loss of generality we can assume that $a=0$. So, since the proof is almost identical to the complex Hessian theory and we have proved a superformalism version of the Lelong-Jensen formula we  give only the lines of the proof. First of all  replacing the $m$-subharmonic function $\widetilde \phi_m(z)=-\frac{1}{(\frac{n}{m}-1)|z|^{2(\frac{n}{m}-1)}}$ used by \cite{14} in the complex Hessian theory  by the corresponding $m$-convex function $\varphi_m(x)=-\frac{1}{(\frac{n}{m}-2)|x|^{\frac{n}{m}-2}}$ if $m\neq\frac{n}{2}$ and $\log|x|$ otherwise. Next, by following almost verbatim the proof of Proposition 2 in \cite{14} and by using Lemma \ref{l1}, we can formulate a variant of the Lelong-Jensen formula similar to that given in Proposition 2 in \cite{14}. Finally, it is not hard to see that such a formula leads to the following conclusion : $$r\longmapsto\frac{1}{r^{\frac{n}{m}(m-n+p)}}\int_{\mathbb{B}(r)\times\mathbb{R}^n}T\wedge\beta^{p}$$ is increases with respect to $r$.
\end{proof}
Theorem \ref{t6} fails when the current $T\wedge\beta^{p-1}$ is concave. Indeed, let $T=-\varphi_m(dd^{\#}\varphi_m)^{m-1}$, $\frac{n}{2}>m$. Then, regarding Example \ref{x1}, it is clear that $T$ is an $m$-positive current ($T$ has locally integrable coefficients) of bidimension $(n-m+1,n-m+1)$ and $T\wedge\beta^{n-m}$ is concave. Again thanks to Example \ref{x1}, a simple computation gives that $r^{\frac{-n}{m}}\int_{\mathbb{B}(r)\times\mathbb{R}^n}T\wedge\beta^{n-m+1}=c_{n,m} r^{\frac{-n}{m}+2}$, for some constant  $c_{n,m}>0$. This means that the $m$-Lelong number of $T$ at the origin does not exist. However, results similar to Theorem \ref{t1} and Theorem \ref{t6},  when $T$ is positive and $T\wedge\omega^{p-1}$ concave, require further conditions. Moving forward, we prove :
\begin{thm}\label{t5} Let $T$ be a weakly positive current of bidimension $(p,p)$ on $\rb^n\times\rb^n$ and $\varphi$ be a ${\mathscr C}^2$ positive function on $\rb^n$ such that $\varphi^\frac{1}{2}$ is convex and $T\wedge\omega^{p-1}$ is concave. If the function  $r\mapsto\frac{\nu_{dd^{\#}T}(\varphi,r)}{2r^{\frac{1}{2}}}$ is integrable in a neighborhood of $0$, then the Lelong number of $T$ relative to the weight $\varphi$ exists.
\end{thm}
Theorem \ref{t5} is a variant of a result obtained by \cite{4} for the negative plurisubharmonic currents in the complex theory. Moreover, as an immediate consequence of Proposition \ref{3}, if $\varphi$ and $T$ are as in Theorem \ref{t1}, then the integrability assumption in Theorem \ref{t5} is clearly satisfied.
\begin{proof} Let $r>0$, and setting
$$\Lambda_{T}(r)=\frac{1}{2^pr^\frac{p}{2}}\int_{B(r)\times\rb^n}T\wedge\omega^p+\frac{1}{2^pr^\frac{p}{2}}\int_{0}^{r}dt\int_{B(t)\times\rb^n}dd^{\#}T\wedge\omega^{p-1}-\int_{0}^{r}\frac{dt}{2^pt^\frac{p}{2}}\int_{B(t)\times\rb^n}dd^{\#}T\wedge\omega^{p-1}.$$
By the integrability condition of $r\mapsto\frac{\nu_{dd^{\#}T}(\varphi,r)}{2r^{\frac{1}{2}}}$ in a neighborhood of  $0$, the function $\Lambda_T$ is well defined and positive on $\rb_+$. Moreover, $$\begin{array}{lcl}
\Lambda_{T}(r)&=&\ds\frac{1}{2^pr^\frac{p}{2}}\int_{B(r)\times\rb^n}T\wedge\omega^p+\int_{0}^{r}\left(\left(\frac{t}{r}\right)^\frac{p}{2}-1\right)\frac{1}{2t^{\frac{1}{2}}}\left[\frac{1}{2^{p-1}t^\frac{p-1}{2}}\int_{B(t)\times\rb^n}dd^{\#}T\wedge\omega^{p-1}\right]dt\\&=&\ds\frac{1}{2^pr^\frac{p}{2}}\int_{B(r)\times\rb^n}T\wedge\omega^p+\int_{0}^{r}\left(\left(\frac{t}{r}\right)^\frac{p}{2}-1\right)\frac{\nu_{dd^{\#}T}(\varphi,t)}{2t^{\frac{1}{2}}}dt
\end{array}$$
On the other hand, in view of Proposition \ref{3}, for every $r_2>r_1>0$, we get
$$\begin{array}{lcl}
\Lambda_{T}(r_2)-\Lambda_{T}(r_1)&=&\ds\frac{1}{2^p{r_2}^\frac{p}{2}}\int_{B(r_2)\times\rb^n}T\wedge\omega^p-\frac{1}{2^p{r_1}^\frac{p}{2}}\int_{B(r_1)\times\rb^n}T\wedge\omega^p\\&+&\ds\frac{1}{2^p{r_2}^\frac{p}{2}}\int_{0}^{r_2}dt\int_{B(t)\times\rb^n}dd^{\#}T\wedge\omega^{p-1}-\frac{1}{2^p{r_1}^\frac{p}{2}}\int_{0}^{r_1}dt\int_{B(t)\times\rb^n}dd^{\#}T\wedge\omega^{p-1}\\&-&\ds\int_{r_1}^{r_2}\frac{dt}{2^pt^\frac{p}{2}}\int_{B(t)\times\rb^n}dd^{\#}T\wedge\omega^{p-1}\\&=&\ds\int_{ B(r_1,r_2)\times\rb^n}T\wedge\alpha^p\geqslant0.
\end{array}$$
Consequently, $\Lambda_T$ is an increasing function on $\rb_+$ and therefore $\ds\lim_{r\rightarrow0}\Lambda_T(r)$ exists. Next, by the integrability condition of $r\mapsto\frac{\nu_{dd^{\#}T}(\varphi,r)}{2r^{\frac{1}{2}}}$ in a neighborhood of  $0$ and
since $t\mapsto\left(\frac{t}{r}\right)^\frac{p}{2}-1$ is uniformly bounded, we have
$$\lim_{r\rightarrow0}\int_{0}^{r}\left(\left(\frac{t}{r}\right)^\frac{p}{2}-1\right)\frac{\nu_{dd^{\#}T}(\varphi,t)}{2t^{\frac{1}{2}}}dt=0.$$
It follows that, $\ds\lim_{r\rightarrow0}\Lambda_T(r)=\lim_{r\rightarrow0}\nu_T(\varphi,r)=\nu_T(\varphi)$.
\end{proof}
Denote by ${\mathscr H}_p$ the $p$-dimensional Hausdorff measure and by $\Supp T$ the support of a given current $T$. By using an integration by part, Proposition 3.2 in \cite{5} and Corollary \ref{c2}, we obtain the following result which is analogue to an elementary one in the complex setting.
\begin{pro}\label{0} Let $T$ be a positive current of bidimension $(p,p)$ such that $p\geqslant 1$,
\begin{enumerate}
\item If $T\wedge\beta^{p-1}$ is convex or concave with compact support, then $T=0$.
\item Assume that $T\wedge\beta^{p-1}$ is convex and let $K$ be a compact subset of $\rb^n$. If ${\mathscr H}_p(K\cap\Supp T)=0$, then $\|T\|_K=0$.
\end{enumerate}
\end{pro}
Note that Proposition \ref{0} improves a result of \cite{5} for positive closed currents. Moreover, the hypothesis $p\geqslant 1$ is necessary, as shown by the positive closed current $(dd^{\#}|x|)^n$ of bidimension $(0,0)$, which is supported by $\{0\}$, but $(dd^{\#}|x|)^n\neq0$.
 \begin{proof}\
 $(1)$ Assume  that $\Supp T=L$ and let $\chi$ be a smooth function such that $0\leqslant\chi\leqslant1$ and $\chi=1$ on $L$, and let $A>0$ so that $|x|^2<A$ on $L$. Then, if $dd^{\#}T\wedge\beta^{p-1}\geqslant0$, an integration by part yields
$$\begin{array}{lcl}
0&\leqslant&\ds\int_{L\times\rb^n}T\wedge(dd^\#|x|^2)^p\leqslant\ds\int_{\rb^n\times\rb^n}\chi T\wedge(dd^\#(|x|^2-A))^p\\&=&\ds\int_{\rb^n\times\rb^n}(|x|^2-A)dd^{\#}(\chi T)\wedge\beta^{p-1}\\&=&\ds\int_{\rb^n\times\rb^n}(|x|^2-A)\left(dd^{\#}\chi\wedge T-d^{\#}\chi\wedge dT+d\chi\wedge d^{\#}T+\chi dd^{\#}T\right)\wedge\beta^{p-1}\\&=&\ds\int_{L\times\rb^n}(|x|^2-A)dd^{\#}T\wedge\beta^{p-1}\leqslant 0.
\end{array}$$
It follows by \cite{5} that $T=0$. On the other hand, when $dd^{\#}T\wedge\beta^{p-1}\leqslant0$ it suffices to rewrite the last integrals with the constant $A=0$.\\
$(2)$ By assumption, we can find a finite number of balls $\mathbb{B}(a_{1},r_{1}),...,\mathbb{B}(a_{N},r_{N})$ such that $K\cap\Supp T\subset\cup_{j=1}^{N}\mathbb{B}(a_{j},r_{j})$ and $\sum_{j=1}^{N}r_{j}^{p}\leqslant\varepsilon$. Thanks to Corollary \ref{c2}, we have
$$\frac{1}{r_{j}^p}\int_{\mathbb{B}(a_j,r_j)\times\mathbb{R}^n}T\wedge\beta^{p}\leqslant \int_{\mathbb{B}(a_j,1)\times\mathbb{R}^n}T\wedge\beta^{p}\leqslant \int_{K_1\times\mathbb{R}^n}T\wedge\beta^{p},$$ where $K_1$ is a compact subset such that $K\cap\Supp T\subset\cup_{j=1}^{N}\mathbb{B}(a_{j},1)\subset K_1$. Hence, if we choose $C=\int_{K_1\times\mathbb{R}^n}T\wedge\beta^{p}$, we get the inequality $$\int_{\mathbb{B}(a_j,r_j)\times\mathbb{R}^n}T\wedge\beta^{p}\leqslant r_{j}^{p}C,\qquad\forall 1\leqslant j\leqslant N.$$ It follows that $$\int_{K\times\mathbb{R}^n}T\wedge\beta^{p}\leqslant C\sum_{i=1}^{N}r_{j}^{p}\leqslant C\varepsilon,$$
and therefore, by arbitrariness of $\varepsilon>0$, we obtain $\|T\|_K=0$.
\end{proof}
 \begin{pro}\label{4} Let $(T_k)_k$ be a sequence of weakly positive closed currents of bidimension $(p,p)$ on $\rb^n\times\rb^n$ which converges weakly to $T$. Then, for any $\mathscr C^2$ positive function $\varphi$ such that $\varphi^\frac{1}{2}$ is convex, we have $$\limsup_{k\rightarrow+\infty}\nu_{T_k}(\varphi)\leqslant\nu_{T}(\varphi).$$
\end{pro}
\begin{proof} For a fixed real $\varepsilon>0$ and $r>0$, let  $\chi_\varepsilon$ be a smooth function such that $0\leqslant\chi_\varepsilon\leqslant1$ and $\chi_\varepsilon=1$ on $\mathbb{B}(r+\frac{\varepsilon}{2})$. Then,
$$\nu_{T_k}(\varphi)\leqslant\frac{1}{2^pr^\frac{p}{2}}\int_{\mathbb{B}(r)\times\rb^n}T_k\wedge\omega^{p}\leqslant\frac{1}{2^pr^\frac{p}{2}}\int_{\mathbb{B}(r+\varepsilon)\times\rb^n}\chi_\varepsilon T_k\wedge\omega^{p}.$$
Since $\chi_\varepsilon(dd^{\#}\varphi)^{p}$ is smooth and with compact support and since $T_k$ converges to $T$ in the sense of currents, by Proposition \ref{2} for $m=n$, we have
$$\limsup_{k\rightarrow+\infty}\nu_{T_k}(\varphi)\leqslant\frac{1}{2^pr^\frac{p}{2}}\int_{\mathbb{B}(r+\varepsilon)\times\rb^n}\chi_\varepsilon T\wedge\omega^{p}.$$
The proof is completed by letting $\varepsilon\rightarrow0$ and $r\rightarrow0$ in this order.
\end{proof}
\section{Comparison theorems and degree of positive supercurrents}
Since the Lelong number relative to a weight $\varphi$ of a weakly positive current $T$ such that $T\wedge\omega^{p-1}$ is convex (or concave) has already been defined, a natural question arises : what's the behaviour of  $\nu_T(\varphi)$ near the set $\varphi^{-1}(0)\cap\Supp T$. In this section we are concerned with the case when $T$ is weakly positive and closed. We obtain an analogue of the famous comparison theorem of  Demailly in the complex setting \cite{3}. More precisely, we have :
\begin{thm}\label{t2} Let $T$ be a weakly positive closed current of bidimension $(p,p)$ on $\rb^n\times\rb^n$. Assume that $\varphi$ and $\psi$ are two $\mathscr C^2$ positive functions on $\rb^n$ such that $\varphi^\frac{1}{2}$ and $\psi^\frac{1}{2}$ are convex and
$$0<l:=\limsup\frac{\psi(x)}{\varphi(x)}\quad as\quad x\in\Supp T\quad and\quad \varphi(x)\rightarrow 0.$$
Then $\nu_T(\psi)\leqslant l^p\nu_T(\varphi)$. In particular, if  $l=\ds\lim\frac{\psi}{\varphi}$ then $\nu_T(\psi)=l^p\nu_T(\varphi)$.
\end{thm}
\begin{proof} By Definition \ref{.1}, we have $\nu_T(\lambda\varphi)=\lambda^p\nu_T(\varphi),\ \forall\lambda>0$. Hence, it suffices to prove that $\nu_T(\psi)\leqslant\nu_T(\varphi)$ when $l<1$. Let's consider the positive convex function  $$u_c=\max(\psi+c,\varphi),\qquad \forall c>0.$$
We have $l<1$, then there exists $t_0>0$ such that $\ds\sup_{\{\varphi<t_0\}}\frac{\psi}{\varphi}<1$. Moreover, let $0<a<r<t_0$ be fixed. Then, for $c>0$ small enough, it is not hard to see that $u_c=\varphi$ on $\varphi^{-1}([a,r])$ and by Stokes formula we obtain
$$\nu_T(\varphi,r)=\nu_T(u_c,r)\geqslant\nu_T(u_c).$$
On the other hand, for any $c>0$, there exists $r>0$ such that $u_c=\psi+c$ on $\{u_c<r\}\cap\Supp T$. It follows that $\nu_T(u_c)=\nu_T(\psi+c)=\nu_T(\psi)$.
Consequently, $\nu_T(\psi)\leqslant\nu_T(\varphi)$. Hence, the equality case is obtained by reversing the role of $\varphi$ and $\psi$ and by observing that $\lim\frac{\varphi}{\psi}=\frac{1}{l}$.
\end{proof}
\begin{thm} Let $T$ be a weakly positive closed current of bidimension $(p,p)$ on $\rb^n\times\rb^n$. Assume that $u_1,...,u_q$ and $v_1,...,v_q$ are convex positive functions and $\varphi$ is a $\mathscr C^2$ positive function on $\rb^n$ such that $\varphi^\frac{1}{2}$ is convex. Suppose that $u_j=0$ on $\Supp T\cap\varphi^{-1}(\lbrace0\rbrace)$ for any $1\leqslant j\leqslant q,$ and that
 $$0<l_j:=\limsup\frac{v_j(x)}{u_j(x)},\ \forall1\leqslant j\leqslant q \quad as\quad x\in\Supp T\quad and\quad \varphi(x)\rightarrow 0.$$
Then, $\ds\nu_{T\wedge dd^{\#}v_1\wedge...\wedge dd^{\#}v_q}(\varphi)\leqslant l_1...l_q\ \nu_{T\wedge dd^{\#}u_1\wedge...\wedge dd^{\#}u_q}(\varphi)$.
\end{thm}
This theorem is a superformalism counterpart of the second comparison theorem of Demailly for the  Lelong number in the complex case \cite{3}.
\begin{proof} Since $dd^{\#}\lambda v_j=\lambda dd^{\#}v_j,\ \forall\lambda>0$, it suffices to give the proof for $l_j<1$. Let's consider the positive convex function $$w_{j,c}=\max\left(v_j+\frac{1}{c},u_j\right),\qquad\forall c>0.$$
We have $l_j<1$, then there exists $t_j>0$ such that $\ds\sup_{\{\varphi<t_j\}}\frac{v_j}{u_j}<1$. For every $c>0$ we can find $r>0$ such that $w_{j,c}=v_j+\frac{1}{c}$ on the set $\{\varphi<r\}\cap\Supp T$. This implies that
$$\nu_{T\wedge dd^{\#}v_1\wedge...\wedge dd^{\#}v_q}(\varphi)=\nu_{T\wedge dd^{\#}w_{1,c}\wedge...\wedge dd^{\#}w_{q,c}}(\varphi).$$
On the other hand, by Proposition \ref{2} for $m=n$, $T\wedge dd^{\#}w_{1,c}\wedge...\wedge dd^{\#}w_{q,c}$ is a sequence of weakly positive closed currents which converges weakly to $T\wedge dd^{\#}u_1\wedge...\wedge dd^{\#}u_q$ when $c\rightarrow+\infty$. Next, in view of Proposition \ref{4}, we get
$$\limsup_{c\rightarrow+\infty}\nu_{T\wedge dd^{\#}w_{1,c}\wedge...\wedge dd^{\#}w_{q,c}}(\varphi)\leqslant \nu_{T\wedge dd^{\#}u_1\wedge...\wedge dd^{\#}u_q}(\varphi).$$
Consequently, $\nu_{T\wedge dd^{\#}v_1\wedge...\wedge dd^{\#}v_q}(\varphi)\leqslant\nu_{T\wedge dd^{\#}u_1\wedge...\wedge dd^{\#}u_q}(\varphi)$.\end{proof}
Similarly as in the complex context, we consider a particular interesting class of convex functions. It is the Lelong class introduced and investigated by \cite{5}, and defined by :  $${\cal L}:=\lbrace f:\mathbb{R}^n\longrightarrow\mathbb{R};\,f(x)\leqslant C| x|+D,\,f\,\mathrm{convex},\,C\geqslant0,\,D\in\mathbb{R}\rbrace.$$ If $f$ is a function in the class ${\cal L}$ then $f$ grows at most linearly at infinity.
\begin{defn}\
\label{.2}
\begin{enumerate}
\item We define the degree of a weakly positive current $T$ of bidimension $(p,p)$ on $\rb^n\times\rb^n$ by $$\delta(T)=\ds\int_{\rb^n\times\rb^n}T\wedge(dd^{\#}|x|)^{p}.$$
\item We say that a function $f$ is semi-exhaustive on the set $E$ if there exists $R$ such that $\{f<R\}\cap E\Subset\rb^n$, and it said to be exhaustive if the condition is fulfilled for every $R$.
\end{enumerate}
\end{defn}
\begin{rem} It was proved by Lagerberg \cite{5} that for every $f_1,...,f_p\in{\cal L}$, the weakly positive current $dd^{\#}f_1\wedge...\wedge dd^{\#}f_p$ is of finite degree. Moreover, as an immediate consequence of Proposition \ref{3}, if $T$ is a weakly positive current such that $dd^{\#}T=0$, then $\delta(T)<+\infty$ if and only if there exists a constant $C>0$ such that $$\nu_T(0,r)=\frac{1}{r^{p}}\int_{\mathbb{B}(r)\times\mathbb{R}^n}T\wedge\beta^{p}\leqslant C,\qquad\forall r>0.$$
\end{rem}
 The following result clarifies the link between the growths of the two quantities $\nu_T(0,r)$ and $\nu_{dd^{\#}T}(0,r)$, where T is a weakly positive current of bidimension $(p,p)$ such that  $T\wedge\beta^{p-1}$ is either convex or concave.
\begin{pro} Let $T$ be a weakly positive current of bidimension $(p,p)$ on $\rb^n\times\rb^n$. Assume that $T\wedge\beta^{p-1}$ is either convex or concave, then we have the growth estimate : $$\nu_{dd^{\#}T}(0,r)=\mathcal{O}\left(r^{-1}\nu_T(0,2r)\right).$$ In particular, if $\nu_T$ has at most linear growth then $dd^{\#}T$ is of finite degree.
\end{pro}
Comparing with the complex setting, notice that there is a clear difference with our case. In fact, in the complex case, if $\nu_T$ is bounded then $\nu_{dd^cT}$ is also bounded, in other wards $dd^cT$ is of finite degree (see Proposition 3.2 in \cite{30}), while in our frame we obtain a more precise estimate : if $\nu_T$ is bounded then $\nu_{dd^{^\#}T}(0,r)$ growth at most like $\frac{1}{r}$.
\begin{proof} Assume that $T$ is concave and consider $\chi$ to be a continuous compactly supported function on $\rb$ such that $\chi(t)=1$ if $|t|\leqslant1$ and $\chi(t)=0$ if $|t|\geqslant2$. By Stokes formula, we have :
$$\begin{array}{lcl}
\nu_{dd^{\#}T}(0,r)&=&\ds\frac{1}{(2r)^{p-1}}\int_{\mathbb{B}(0,r)\times\rb^n}dd^{\#}T\wedge\beta^{p-1}\\&\geqslant&\ds\frac{1}{(2r)^{p-1}}\int_{\mathbb{B}(0,2r)\times\rb^n}dd^{\#}T\wedge\chi\left(\frac{|x|^2}{2r^2}\right)\beta^{p-1}\\&=&\ds\frac{1}{(2r)^{p-1}}\int_{\mathbb{B}(0,2r)\times\rb^n}T\wedge dd^{\#}\chi\left(\frac{|x|^2}{2r^2}\right)\wedge\beta^{p-1}\\&=&\ds\frac{1}{r(2r)^{p}}\int_{\mathbb{B}(0,2r)\times\rb^n}\chi'\left(\frac{|x|^2}{2r^2}\right)T\wedge\beta^p\\&+&\ds\frac{1}{r(2r)^{p}}\int_{\mathbb{B}(0,2r)\times\rb^n}\chi''\left(\frac{|x|^2}{2r^2}\right)T\wedge\frac{d|x|^2\wedge d^{\#} |x|^2}{2r^2}\wedge\beta^{p-1}.
\end{array}$$
As $|\chi'|$ and $|\chi''|$ are bounded, and $d|x|^2\wedge d^{\#} |x|^2\leqslant 2|x|^2 dd^{\#}|x|^2$, we deduce the following estimates $$r\nu_{dd^{\#}T}(0,r)\geqslant-c_1\nu_{T}(0,2r)-c_2\nu_{T}(0,2r)\geqslant-c\nu_{T}(0,2r).$$
In the case of a convex current, we just reverse the above inequalities.
\end{proof}
Now, by following a result given by Elkhadhra and Mimouni \cite{30}, we establish that a weakly positive current $T$ such that $dd^{\#}T=0$ is of finite degree provided that his support is contained in a strip. More precisely we prove :
\begin{thm} Let $T$ be a weakly positive current of bidimension $(p,p)$ on $\rb^n\times\rb^n$. Assume that $T$ is concave and $\Supp T\subset\{|x_{k+1}|^\delta+...+|x_n|^\delta\leqslant1\}$ for some $\delta\in\nb^\ast$ and for $p\geqslant k$. Then, there exists a constant $C>0$ such that for all $r\geqslant1$ we have $\nu_T(0,r)\leqslant C$. In particular, if $dd^{\#}T=0$, then $T$ is of finite degree.
\end{thm}
\begin{proof} By considering the current $T\wedge\beta^{p-k}$, we may assume that $p=k$. Let's begin with the case $T$ is smooth. Let $\chi$ be a $\mathscr{C}^\infty$ function such that $\chi(t)=1$ if $|t|\leqslant1$ and equal to $0$ if $|t|>2$. Let $\beta'=\frac{1}{2}dd^{\#}|x'|^2$ for all $x'=(x_1,...,x_p)\in\rb^p$, and for $a=(a_1,...,a_p)\in\rb^p$ let us denote $g(a)=\int_{\rb^n\times\rb^n}T\wedge\chi(|x-a|^2)\beta'^p$. Then,
$$\begin{array}{lcl}
\ds\frac{\partial^2g}{\partial a_{1}^2}&=&\ds\int_{\rb^n\times\rb^n}T\wedge\frac{\partial^2}{\partial a_{1}^2}\chi(|x'-a|^2)\beta'^p\\&=&\ds\int_{\rb^n\times\rb^n}T\wedge\frac{\partial^2}{\partial x_{1}^2}\chi(|x'-a|^2)\beta'^p\\&=&\ds\int_{\rb^n\times\rb^n}T\wedge dd^{\#}\left(\chi(|x'-a|^2)dx_2\wedge d\xi_2\wedge...\wedge dx_p\wedge d\xi_p\right)\\&=&\ds \int_{\rb^n\times\rb^n}\chi(|x'-a|^2)dd^{\#}T\wedge dx_2\wedge d\xi_2\wedge...\wedge dx_p\wedge d\xi_p\leqslant0.
\end{array}$$
Thus, the function $a_1\mapsto -g(a_1,a_2,..,a_p)$ is negative and convex on $\rb$ and therefore $g$ is constant with respect to $a_1$. By iteration, we see that $g$ is constant, i.e. $g(a)=g(0)=\int_{\rb^n\times\rb^n}T\wedge\chi(|x'|^2)\beta'^p$. Hence, there exists a constant $C>0$ such that $\int_{\{|x'|\leqslant1,x''\}\times\rb^n}T\wedge\beta'^p\leqslant C$, where $x''=(x_{p+1},...,x_n)$. Let $j\in\{p+1,...,n\}$, then
$$\begin{array}{lcl}
\ds\int_{\rb^n\times\rb^n}T\wedge\chi^2(|x'|^2)dd^{\#}|x_j|^2\wedge\beta'^{p-1}&=&\ds\int_{\rb^n\times\rb^n}T\wedge dd^{\#}\left(|x_j|^2\chi^2(|x'|^2)\right)\wedge\beta'^{p-1}\\&-&\ds\int_{\rb^n\times\rb^n}T\wedge|x_j|^2dd^{\#}\chi^2(|x'|^2)\wedge\beta'^{p-1}\\&-&2\ds\int_{\rb^n\times\rb^n}T\wedge d\chi^2(|x'|^2)\wedge d^{\#}|x_j|^2\wedge\beta'^{p-1}\\&=&(1)+(2)+(3).
\end{array}$$
By Stokes formula and the fact that $|x_j|^2\chi^2(|x'|^2)\beta'^{p-1}$ has compact support on $\Supp T$, we get
$$(1)\leqslant\int_{\rb^n\times\rb^n}|x_j|^2\chi^2(|x'|^2)dd^{\#}T\wedge\beta'^{p-1}\leqslant0.$$
On the other hand, since $|x_j|$ is bounded on $\Supp T$ and $|\chi|$, $|\chi'|$ and $|\chi''|$ are bounded, there exists a constant $C>0$ such that $$(2)=-\int_{\rb^n\times\rb^n}T\wedge|x_j|^2dd^{\#}\chi^2(|x'|^2)\wedge\beta'^{p-1}\leqslant C\int_{\{1\leqslant|x'|\leqslant2,x''\}\times\rb^n}T\wedge\beta'^{p}\leqslant C_1.$$
To obtain $C_1$, we may slightly modify $\chi$ by taking $\chi(t)=1$ if $|t|\leqslant2$ and $0$ if
$|t|>3$ and repeat the above argument. Let $\varphi$ be a smooth and compactly supported function on $\rb$ such that $0\leqslant\varphi\leqslant1$ and $\varphi=1$ on $\Supp\chi$. By the Cauchy-Schwarz inequality, we have
$$\begin{array}{lcl}
|(3)|&\leqslant&\ds\bigg\vert\int_{\rb^n\times\rb^n}T\wedge 4\chi(|x'|^2)\varphi(|x'|^2)d\chi(|x'|^2)\wedge d^{\#}|x_j|^2\wedge\beta'^{p-1}\bigg\vert\\&\leqslant&\ds\frac{1}{\varepsilon}\int_{\rb^n\times\rb^n}T\wedge 4\varphi^2(|x'|^2)d\chi(|x'|^2)\wedge d^{\#}\chi(|x'|^2)\wedge\beta'^{p-1}\\&+&\ds\varepsilon\int_{\rb^n\times\rb^n}T\wedge 4\chi^2(|x'|^2)d|x_i|^2\wedge d^{\#}|x_j|^2\wedge\beta'^{p-1}\\&\leqslant&\ds\frac{C}{\varepsilon}\int_{\{1\leqslant|x'|\leqslant2,x''\}\times\rb^n}T\wedge\beta'^{p}+8\varepsilon\int_{\rb^n\times\rb^n}T\wedge\chi^2(|x'|^2)dd^{\#}|x_j|^2\wedge\beta'^{p-1}\\&\leqslant&\ds\frac{C_2}{\varepsilon}+8\varepsilon\int_{\rb^n\times\rb^n}T\wedge\chi^2(|x'|^2)dd^{\#}|x_j|^2\wedge\beta'^{p-1}.
\end{array}$$
Choosing $\varepsilon=\frac{1}{16}$, we obtain
$$\int_{\rb^n\times\rb^n}T\wedge\chi^2(|x'|^2)dd^{\#}|x_j|^2\wedge\beta'^{p-1}\leqslant C_1+16C_2+\frac{1}{2}\int_{\rb^n\times\rb^n}T\wedge\chi^2(|x'|^2)dd^{\#}|x_j|^2\wedge\beta'^{p-1}.$$
Taking $C_3=2(C_1+16C_2)$, then since $dd^{\#}|x''|^2=\sum_{j=p+1}^ndd^{\#}|x_j|^2$, we have
$$\int_{\rb^n\times\rb^n}T\wedge\chi^2(|x'|^2)dd^{\#}|x''|^2\wedge\beta'^{p-1}\leqslant (n-p)C_3.$$
In order to show that the integral $\int_{\rb^n\times\rb^n}T\wedge\chi^2(|x'|^2)(dd^{\#}|x''|^2)^2\wedge\beta'^{p-2}$ is finite, we use the last inequality and we rewrite the previous proof with $\beta'^{p-1}$ replaced by $dd^{\#}|x''|^2\wedge\beta'^{p-1}$. Proceeding by induction, we show
that there exists a constant $C_4>0$ such that for $1\leqslant s\leqslant p$,
$$\int_{\rb^n\times\rb^n}T\wedge\chi^2(|x'|^2)(dd^{\#}|x''|^2)^s\wedge\beta'^{p-s}\leqslant C_4.$$
It follows that there exists $C_5>0$ such that
$$\int_{\{|x'|\leqslant1,|x''|\leqslant1\}\times\rb^n}T\wedge\beta^{p}\leqslant \int_{\rb^n\times\rb^n}T\wedge\chi^2(|x'|^2)\beta^{p}\leqslant C_5.$$
Now, let us assume that $T$ is not smooth. Let $T_\varepsilon$ be a regularization of $T$ and $g_\varepsilon$ be the function associated with $T_\varepsilon$. The sequence $T_\varepsilon$ converges weakly to $T$, and it is easy to see that the
sequence $g_\varepsilon(a)$ tends to $g(a)$. By the above discussion we find that $g_\varepsilon$ is constant with respect to $a$, so as well as $g$, and therefore $\int_{\{|x'|\leqslant1,x''\}\times\rb^n}T\wedge\beta'^{p}\leqslant C$. For $r>1$, we can cover $\{|x'|<r,|x''|\leqslant1\}$ by at most $([r]+1)^p$ unit cubes, where $[r]$ denotes the integer part of $r$.
Thus, $$\int_{\mathbb{B}(0,r)\times\rb^n}T\wedge\beta^{p}\leqslant([r]+1)^pC_5,$$ and the desired result will follows. If $dd^{\#}T=0$, then the variant of the Lelong-Jensen formula in the sperformalism setting implies that $T$ is of finite degree.
\end{proof}
\begin{exe} The hypothesis $T$ is concave in the previous theorem is necessary as shown the following example :
let $f$ and $g$ be two smooth compactly supported and positive functions on the interval $]-1,1[$ such that $g(x_2)dd^{\#}|x_2|^2+dd^{\#}f(x_2)\geqslant0$, and let $$T=f(x_2)dd^{\#}|x_1|^2+g(x_2)|x_1|^2dd^{\#}|x_2|^2.$$
It is clear that $T$ is a weakly positive convex current of bidegree $(1,1)$ on $\rb^2\times\rb^2$ and with support in the strip $\{(x_1,x_2)\in\rb^2;\ |x_2|\leqslant1\}$, but $\nu_T(0,r)$ is not bounded. \end{exe}
In the next result, we establish a version of the comparison theorem of Rashkovskii \cite{6} in the superformalism setting.
\begin{thm}\label{t3} Assume that $T$ is a weakly positive closed current of bidimension $(p,p)$ on $\rb^n\times\rb^n$ and of finite degree. Let $u_1,...,u_p\in{\cal L}$, and let $v_1,...,v_p\in{\cal L}$ are semi-exhaustive on $\Supp T$. Suppose that for every $\eta>0$ and $1\leqslant j\leqslant p$, we have
$$l_j\geqslant\limsup\frac{u_j(x)}{v_j(x)+\eta|x|}\quad as\quad x\in\Supp T\quad and\quad |x|\rightarrow+\infty.$$
Then $\ds\int_{\mathbb{R}^n\times\mathbb{R}^n}T\wedge dd^{\#}u_1\wedge...\wedge dd^{\#}u_p\leqslant l_1...l_p\ds\int_{\mathbb{R}^n\times\mathbb{R}^n}T\wedge dd^{\#}v_1\wedge...\wedge dd^{\#}v_p$.
\end{thm}
\begin{proof} It suffices to prove that the condition
\begin{equation}\label{e2}
1>\limsup\frac{u_j(x)}{v_j(x)+\eta|x|}\quad as\quad x\in\Supp T\quad and\quad |x|\rightarrow +\infty,\quad\forall\eta>0,\,1\leqslant j\leqslant p.
\end{equation}
imply
$$
\ds\int_{\mathbb{R}^n\times\mathbb{R}^n}T\wedge dd^{\#}u_1\wedge...\wedge dd^{\#}u_p\leqslant\ds\int_{\mathbb{R}^n\times\mathbb{R}^n}T\wedge dd^{\#}v_1\wedge...\wedge dd^{\#}v_p.$$
By virtue of (\ref{e2}), for every $C>0$, there exists $0<\alpha_j=\alpha_j(C,\eta,u_j,v_j)$ such that
$$E_j(C)=\{x\in\Supp T;\ v_j(x)+\eta|x|-C<u_j(x)\}\Subset \mathbb{B}(\alpha_j).$$
Setting $\alpha=\max_j(\alpha_j)$, $E(C)=\cap_jE_j(C)$ and $$w_{j,C}=\max\{v_j(x)+\eta|x|-C,u_j\}.$$ Since $w_{j,C}=v_j(x)+\eta|x|-C$ in a neighborhood of $\partial \mathbb{B}(\alpha)\cap\Supp T$, we obtain
$$\begin{array}{lcl}
\ds\int_{\mathbb{B}(\alpha)\times\mathbb{R}^n}T\wedge dd^{\#}w_{1,C}\wedge...\wedge dd^{\#}w_{p,C}&=&\ds\int_{\mathbb{B}(\alpha)\times\mathbb{R}^n}T\wedge dd^{\#}(v_1+\eta|x|)\wedge...\wedge dd^{\#}(v_p+\eta|x|)\\&\leqslant&\ds\int_{\mathbb{R}^n\times\mathbb{R}^n}T\wedge dd^{\#}(v_1+\eta|x|)\wedge...\wedge dd^{\#}(v_p+\eta|x|).
\end{array}$$
Observe that for every compact set $K$ of $\rb^n$, we can find a constant $C_K>0$ such that $K\cap\Supp T\subset E(C)$ for any $C>C_K$. It follows that for $R>0$ and $C>C_R$, we have $$\ds\int_{\mathbb{B}(R)\times\mathbb{R}^n}T\wedge dd^{\#}w_{1,C}\wedge...\wedge dd^{\#}w_{p,C}\leqslant\ds\int_{\mathbb{R}^n\times\mathbb{R}^n}T\wedge dd^{\#}(v_1+\eta|x|)\wedge...\wedge dd^{\#}(v_p+\eta|x|).$$
 On the other hand, for every $1\leqslant j\leqslant p$, the sequence of convex functions $(w_{j,s})_{s}$ is decreasing to $u_j$, then by using Proposition \ref{2}, we get the following weak convergence :
 $$T\wedge dd^{\#}w_{1,s}\wedge...\wedge dd^{\#}w_{p,s}\longrightarrow T\wedge dd^{\#}u_1\wedge...\wedge dd^{\#}u_p,\qquad\mathrm{when}\,s\rightarrow+\infty.$$
Consequently,
$$\begin{array}{lcl}
\ds\int_{\mathbb{B}(R)\times\mathbb{R}^n}T\wedge dd^{\#}u_1\wedge...\wedge dd^{\#}u_p&\leqslant&\ds\limsup_{s\rightarrow+\infty}\int_{\mathbb{B}(R)\times\mathbb{R}^n}T\wedge dd^{\#}w_{1,s}\wedge...\wedge dd^{\#}w_{p,s}\\&\leqslant&\ds\int_{\mathbb{R}^n\times\mathbb{R}^n}T\wedge dd^{\#}(v_1+\eta|x|)\wedge...\wedge dd^{\#}(v_p+\eta|x|).
\end{array}$$
Since $\delta(T)<+\infty$, an adaptation of the proof of  Proposition 3.10 in \cite{5} yields
$$\int_{\mathbb{R}^n\times\mathbb{R}^n}T\wedge dd^{\#}f_1\wedge...\wedge dd^{\#}f_p<+\infty,\qquad\forall f_1,...,f_p\in{\cal L}.$$ Therefore, by arbitrariness of $\eta$, we obtain the following inequality
$$\int_{\mathbb{B}(R)\times\mathbb{R}^n}T\wedge dd^{\#}u_1\wedge...\wedge dd^{\#}u_p\leqslant\int_{\mathbb{R}^n\times\mathbb{R}^n}T\wedge dd^{\#}v_1\wedge...\wedge dd^{\#}v_p.$$
The proof is completed by letting $R$ tends to $+\infty$.
\end{proof}
As an immediate consequence of Theorem \ref{t3}, we obtain :
\begin{cor}\label{c3} Let $u_1,...,u_p$ and $T$ as in Theorem \ref{t3}, then $$\int_{\mathbb{R}^n\times\mathbb{R}^n}T\wedge dd^{\#}u_1\wedge...\wedge dd^{\#}u_p\leqslant \delta(T)\sigma(u_1)...\sigma(u_p),$$ where $\sigma(u_j)=\ds\limsup\frac{u_j(x)}{|x|}$ as $x\in\Supp T$ and $|x|\rightarrow +\infty$, $\forall1\leqslant j\leqslant p$.
\end{cor}
\begin{proof} For every $\eta>0$, we have
$$\limsup\frac{u_j(x)}{|x|+\eta|x|}\leqslant\limsup\frac{u_j(x)}{|x|}=\sigma(u_j)\ \ as\ \  x\in\Supp T\ \ and\ \ |x|\rightarrow +\infty,\ \ \forall1\leqslant j\leqslant p.$$
Then, by Theorem \ref{t3}, we obtain

$\ds\int_{\mathbb{R}^n\times\mathbb{R}^n}T\wedge dd^{\#}u_1\wedge...\wedge dd^{\#}u_p\leqslant\sigma(u_1)...\sigma(u_p)\int_{\mathbb{R}^n\times\mathbb{R}^n}T\wedge(dd^{\#}|x|)^p=\delta(T)\sigma(u_1)...\sigma(u_p).$
\end{proof}
A direct consequence of Corollary \ref{c3} is that we can obtain an infinite number of currents of finite degree just by considering the current $T\wedge dd^{\#}u_1\wedge...\wedge dd^{\#}u_{k},\ \forall 1\leqslant k\leqslant p$, where $T$ is a weakly positive closed current of bidimension $(p,p)$ and of finite degree and  $u_1,...,u_p\in{\cal L}$.
\begin{defn} Let $\varphi$ be a convex function on $\mathbb{R}^n$ and $T$ is a weakly positive current of bidimension $(p,p)$ on $\rb^n\times\rb^n$. We introduce the generalized degree relative to $\varphi$ by the quantity
 $$\delta(T,\varphi)=\ds\int_{\mathbb{R}^n\times\mathbb{R}^n}T\wedge(dd^{\#}\varphi)^{p}.$$
 \end{defn}
In particular, when $\varphi=|x|$, $\delta(T,|x|)=\delta(T)$. In terms of weighted degree, Corollary \ref{c3} can be generalized as follow :
\begin{cor} Let $T$ be a weakly positive closed current of finite degree and of bidimension $(p,p)$ on $\rb^n\times\rb^n$, and let $u_1,...,u_p\in{\cal L}$. Then, for every $\varphi\in{\cal L}$ semi-exhaustive on $\Supp T$, we have  $$\int_{\mathbb{R}^n\times\mathbb{R}^n}T\wedge dd^{\#}u_1\wedge...\wedge dd^{\#}u_p\leqslant \delta(T,\varphi)\sigma(u_1,\varphi)...\sigma(u_p,\varphi),$$ where $\sigma(u_j,\varphi)=\ds\limsup\frac{u_j(x)}{\varphi(x)}$ as $x\in\Supp T$ and $|x|\rightarrow +\infty$, $\forall 1\leqslant j\leqslant p$.
\end{cor}
\begin{proof}  For $\eta>0$, we have $$\limsup\frac{u_j(x)}{\varphi(x)+\eta|x|}\leqslant\limsup\frac{u_j(x)}{\varphi(x)}=\sigma(u_j,\varphi),\ \ as\ \ x\in\Supp T\ \ and\ \ |x|\rightarrow +\infty,\ \ \forall1\leqslant j\leqslant p.$$
Hence, by Theorem \ref{t3}, we obtain\\
$$\int_{\mathbb{R}^n\times\mathbb{R}^n}T\wedge dd^{\#}u_1\wedge...\wedge dd^{\#}u_p\leqslant\sigma(u_1,\varphi)...\sigma(u_p,\varphi)\ds\int_{\mathbb{R}^n\times\mathbb{R}^n}T\wedge(dd^{\#}\varphi)^p$$
$\qquad\qquad\qquad\qquad\qquad\qquad\qquad\qquad\qquad\quad\ds=\delta(T,\varphi)\sigma(u_1,\varphi)...\sigma(u_p,\varphi).$
\end{proof}
The next result is another form of comparison theorem, which is a version of a result due to Coman and Nivoche \cite{1} in the complex category.
\begin{pro} Let $T$ be a weakly positive closed current of bidimension $(p,p)$ on $\rb^n\times\rb^n$, $p\geqslant 1$. Let $\varphi$ and $\psi$ be two convex functions on $\rb^n$ such that $$\ds\lim_{|x|\rightarrow\infty}\varphi(x)=+\infty\quad{\rm and}\quad 0<l:=\ds\limsup\frac{\psi(x)}{\varphi(x)},\quad as\quad x\in\Supp T\quad and\quad |x|\rightarrow +\infty,$$ then $\delta(T,\psi)\leqslant l^{p}\delta(T,\varphi)$.
In particular, if $l=\ds\lim\frac{\psi}{\varphi}$ then $\delta(T,\psi)=l^{p}\delta(T,\varphi)$.
\end{pro}
\begin{proof} For the proof, we proceed as in \cite{1}. Since $\delta(T,\lambda\varphi)=\lambda^{p}\delta(T,\varphi),\ \forall\lambda>0$, it suffices to prove the inequality for $l=1$. For $\varepsilon>0$, $R>0$ and $M>0$ fixed, we put $$\psi_M=\max\{\psi,-M\},\quad w_m=\max\{(1+\varepsilon)\varphi-m,\psi_M\}.$$
For $m$ large enough, $w_m=\psi_M$ on the ball $\mathbb{B}(2R)$. On the other hand, by hypothesis we can find $R'>2R$, such that $w_m=(1+\varepsilon)\varphi-m$ on $\{|x|>R'\}$. Let $\phi$ be a smooth function on $\rb^n$ such that $0\leqslant\phi\leqslant1$ and $\phi=1$ on $\overline{\mathbb{B}}(R')$. Then, Stokes formula gives
$$\begin{array}{lcl}
\ds\int_{\mathbb{B}(2R)\times\rb^n}T\wedge(dd^{\#}\psi_M)^{p}\leqslant\ds\int_{\overline{\mathbb{B}}(R')\times\rb^n}T\wedge(dd^{\#}w_m)^{p}&\leqslant&\ds\int_{\rb^n\times\rb^n}T\wedge\phi(dd^{\#}w_m)^{p}\\&=&\ds\int_{\rb^n\times\rb^n}T\wedge w_mdd^{\#}\phi\wedge(dd^{\#}w_m)^{p-1}.
\end{array}$$
As the support of $dd^{\#}\phi$ is included in the set $\{|x|>R'\}$, where $w_m=(1+\varepsilon)\varphi-m$. Then, by replacing $w_m$ by $(1+\varepsilon)\varphi-m$ and applying another time Stokes formula, the last integral is equals to $(1+\varepsilon)^{p}\ds\int_{\rb^n\times\rb^n}T\wedge\phi(dd^{\#}\varphi)^{p}$. It follows that
$$\int_{\mathbb{B}(2R)\times\rb^n}T\wedge(dd^{\#}\psi_M)^{p}\leqslant(1+\varepsilon)^{p}\ds\int_{\rb^n\times\rb^n}T\wedge(dd^{\#}\varphi)^{p}.$$
Moreover, the sequence $\psi_M$ is convex decreasing to $\psi$, so by Proposition \ref{2} we have the weak convergence $T\wedge (dd^{\#}\psi_M)^{p}\longrightarrow T\wedge (dd^{\#}\psi)^{p}$ as $M\rightarrow+\infty$. Then,
$$\int_{\mathbb{B}(2R)\times\rb^n}T\wedge(dd^{\#}\psi)^{p}\leqslant(1+\varepsilon)^{p}\ds\int_{\rb^n\times\rb^n}T\wedge(dd^{\#}\varphi)^{p}.$$
The proof is finished by letting $\varepsilon\rightarrow0$ and $R\rightarrow+\infty$ in this order.
\end{proof}
We close this section with a version of the semi-continuity results due to Demailly \cite{3} in the superformalism setting.
\begin{pro}\
\begin{enumerate}
\item Assume that $T_k,T$ are weakly positive closed currents of bidimension $(p,p)$ on $\rb^n\times\rb^n$ such that $(T_k)_k$ converges weakly to $T$. Then, for all $\varphi$ a convex and exhaustive function on $\cup_k\Supp T_k$, we have $$\delta(T,\varphi)\leqslant \liminf_{k\rightarrow+\infty}\delta(T_k,\varphi).$$
\item Let $T$ be a weakly positive closed current of bidimension $(p,p)$ on $\rb^n\times\rb^n$. Then, for all sequence $(\varphi_k)_k$ of convex and exhaustive functions on $\Supp T$ which converges pointwise to $\varphi$, we have $$\delta(T,\varphi)\leqslant \liminf_{m\rightarrow+\infty}\delta(T,\varphi_k).$$
\end{enumerate}
\end{pro}
\begin{proof}\
 $(1)$ For $\varepsilon>0$ and $R>0$ fixed, let $(\varphi_m)_m$ be a sequence of convex and smooth functions converges to $\varphi$ such that $\varphi\leqslant\varphi_m<\varphi+\frac{1}{m}$ on $\{R-\varepsilon\leqslant\varphi\leqslant R+\varepsilon\}$, and we put
$$\psi_m=\left\lbrace\begin{array}{lll}
\varphi\quad{\rm on}\ \rb^n\smallsetminus B(R)\\
\max\{\varphi,(1-\varepsilon)(\varphi_m-\frac{1}{m})+R\varepsilon\}\quad{\rm on}\ \bar{B}(R),\\
\end{array}\right.$$
where $B(R)=\{x\in\rb^n;\varphi(x)<R\}$. It is clear that the definition is coherent and $\psi_m$ is convex. Choose a smooth function $\chi_\varepsilon$ such that $0\leqslant\chi_\varepsilon\leqslant1$, $\chi_\varepsilon=1$ on $B(R-\varepsilon)$ and with support in $B(R-\frac{\varepsilon}{2})$. Then, for all $m\geqslant\left[\frac{2(1-\varepsilon)}{\varepsilon^2}\right]$ we have
$$\begin{array}{lcl}
\ds\int_{B(R)\times\rb^n}T_k\wedge(dd^{\#}\varphi)^{p}=\ds\int_{B(R)\times\rb^n}T_k\wedge(dd^{\#}\psi_m)^{p}&\geqslant&\ds\int_{B(R-\frac{\varepsilon}{2})\times\rb^n}T_k\wedge(dd^{\#}\psi_m)^{p}\\&\geqslant&\ds(1-\varepsilon)^{p}\int_{B(R-\frac{\varepsilon}{2})\times\rb^n}\chi_\varepsilon T_k\wedge(dd^{\#}\varphi_m)^{p}.
\end{array}$$
Since $\chi_\varepsilon(dd^{\#}\varphi_m)^{p}$ is smooth and with compact support and $T_k$ converges weakly to $T$, we obtain $$\liminf_{k\rightarrow+\infty}\delta(T_k,\varphi)\geqslant\liminf_{k\rightarrow+\infty}\int_{B(R)\times\rb^n}T_k\wedge(dd^{\#}\varphi)^{p}\geqslant(1-\varepsilon)^{p}\int_{B(R-\frac{\varepsilon}{2})\times\rb^n}\chi_\varepsilon T\wedge(dd^{\#}\varphi_m)^{p}.$$
In virtue of Proposition \ref{2}, we get
$$\liminf_{k\rightarrow+\infty}\delta(T_k,\varphi)\geqslant(1-\varepsilon)^{p}\int_{B(R-\frac{\varepsilon}{2})\times\rb^n}\chi_\varepsilon T\wedge(dd^{\#}\varphi)^{p}.$$
The proof of (1) is finished by letting $\varepsilon\rightarrow0$ and $R\rightarrow+\infty$ in this order.\\
$(2)$ For $R>0$ and $\varepsilon>0$ fixed, let $\chi_\varepsilon$ be a smooth function such that $0\leqslant\chi_\varepsilon\leqslant1$, $\chi_\varepsilon=1$ on $B(R-\varepsilon)$ and with support in $B(R)$. Then, $$\delta(T,\varphi_k)\geqslant\int_{B(R)\times\rb^n}T\wedge(dd^{\#}\varphi_k)^{p}\geqslant\int_{B(R)\times\rb^n}\chi_\varepsilon T\wedge(dd^{\#}\varphi_k)^{p}.$$
By using Proposition \ref{2}, it follows that $$\liminf_{k\rightarrow+\infty}\delta(T,\varphi_k)\geqslant\int_{B(R)\times\rb^n}\chi_\varepsilon T\wedge(dd^{\#}\varphi)^{p}.$$
The proof of (2) is completed by letting $\varepsilon\rightarrow0$ and $R\rightarrow+\infty$ in this order.
\end{proof}
\section{On the extension of positive supercurrents}
In this section, we are interested with the extension of positive currents in the superformalism setting. By an adaptation of the techniques of Dabbek, Elkhadhra and El Mir \cite{31} in the complex context and the work of Berndtsson \cite{11} on the removable singularities of minimal currents, we prove the following theorem :
\begin{thm}\label{t7} Let $\Omega$ be an open subset of $\rb^n$ and  $T$ be a weakly positive current of bidimension $(p,p)$ on
$\{\Omega\smallsetminus K\}\times\rb^n$ with locally finite mass near $K$, where $K$ is a compact subset of $\rb^n$ with sigma-finite $(p-2)$-dimensional Hausdorff measure.  Assume that either the measure $dd^{\#}T\wedge\beta^{p-1}$ is locally finite near $K$ or $T\wedge\beta^{p-1}$ is concave on $\{\Omega\smallsetminus K\}\times\rb^n$, then there exists a positive measure $S$ supported in $K$ such that
$\widetilde{dd^{\#}T\wedge\beta^{p-1}}=dd^{\#}\widetilde{T}\wedge\beta^{p-1}+S,$ where $\widetilde{T}$ and $\widetilde{dd^{\#}T\wedge\beta^{p-1}}$ are the  extensions by $0$ of $T$ and $dd^{\#}T\wedge\beta^{p-1}$ respectively across $K\times\rb^n$.
\end{thm}
As an immediate consequence we see that in the case $T\wedge\beta^{p-1}$ is concave, the current $\widetilde{T}\wedge\beta^{p-1}$ is also concave. Based on the book of Landkof \cite{0} it is easy to see that for any compact set $K$ with sigma-finite $(p-2)$-dimensional Hausdorff measure there exists a potential $u$ such that $u=-\infty$ on $K$ and $u$ is smooth outside $K$. In order to get Theorem \ref{t7}, we prove :
\begin{pro}\label{6} With the same hypothesis of Theorem \ref{t7}, for $\alpha>0$ and ${\cal\rm O}\Subset\Omega$, we have $$\ds\int_{\{{\cal\rm O}\smallsetminus K\}\times\rb^n} T\wedge\beta^{p-1}\wedge{{du\wedge d^{\#} u}\over
u^2(\log-u)^{1+\alpha}}\ <+\infty .$$
\end{pro}
\begin{proof} Let $\psi$ be a smooth, compactly supported, even and positive on $]-1,1[$ such that $\int\psi(t)dt=1$ and let
$\chi_k(t)=\sup(t-{\frac{2}{k}},0)\ast\psi_k(t)$ where $\psi_k(t)=k\psi(kt)$. Here
$\chi_k$ is a sequence of convex increasing functions
converging towards $\sup(t,0)$. Furthermore, $\chi_k(t)=0$ if $t<\frac{1}{k}$ and $\chi'_k(t)\leqslant 1$. We claim that if $g$ is a test function such that $g=1$ in a neighborhood of ${\cal\rm O}$, then, the sequence
$\left\langle T\wedge\beta^{p-1},g^2dd^{\#}\chi_k((\log-u)^{-\alpha})\right\rangle$ is bounded. Indeed, since $T\wedge\beta^{p-1}$ is symmetric of bidimension $(1,1)$, a simple computation proves that $T\wedge\beta^{p-1}\wedge du\wedge d^{\#}g=T\wedge\beta^{p-1}\wedge dg\wedge d^{\#}u$, which leads to
\begin{equation}\label{e8}
\begin{array}{lcl}
\left\langle T\wedge\beta^{p-1}, g^2 dd^{\#}\chi_k
((\log-u)^{-\alpha}) \right\rangle&=&\left\langle dd^{\#}T\wedge\beta^{p-1}, \chi_k((\log-u)^{-\alpha})g^2\right\rangle\\&-& \left\langle T\wedge\beta^{p-1},\chi_k ((\log-u)^{-\alpha}) dd^{\#}g^2 \right\rangle\\&-&
2\left\langle T\wedge\beta^{p-1}, \chi'_k ((\log-u)^{-\alpha}){{2g\alpha du\wedge d^{\#}g}\over(- u)(\log-u)^{\alpha +1}}\right\rangle\\&=&I+II+III.
\end{array}
\end{equation}
Assume that the currents $\widetilde{dd^{\#}T\wedge\beta^{p-1}}$ and $\widetilde{T}$ exist, then $I$ and $II$ are bounded. On the other hand, by considering the following symmetric (because $T$ is symmetric) bilinear form
$$(\varphi,\psi)=\int_{\mathbb{R}^n\times\mathbb{R}^n}T\wedge\beta^{p-1}\wedge\varphi\wedge J(\psi),\quad\quad\forall\,\varphi,\psi\in{\mathscr D}^{1,0}(\mathbb{R}^n\times\mathbb{R}^n),$$
we see that $(\varphi,\varphi)\geqslant0$, since $T$ is weakly positive. Then, the Cauchy-Schwarz inequality yields
\begin{equation}\label{e10}
\left|III\right|\leq {1\over 2} \left\langle T\wedge\beta^{p-1},g^2\chi'^2_k ((\log-u)^{-\alpha}){{du\wedge d^{\#}u}\over u^2(\log-u)^{2\alpha+2}}\right\rangle + 32\alpha^2\left\langle
T\wedge\beta^{p-1} , dg\wedge d^{\#} g\right\rangle.
\end{equation}
Next, a simple computation gives
\begin{equation}\label{e9}
\begin{array}{lcl}
\left\langle T\wedge\beta^{p-1}, g^2dd^{\#}\chi_k ((\log-u)^{-\alpha})
\right\rangle &=& \left\langle T\wedge\beta^{p-1} , \chi'_k ((\log-u)^{-\alpha}){{\alpha
g^2dd^{\#} u}\over (-u)(\log-u)^{\alpha+1}}\right\rangle\\&+&\
\left\langle T\wedge\beta^{p-1} , \chi'_k ((\log-u)^{-\alpha}){{\alpha g^2
du\wedge
d^{\#}u}\over u^2(\log-u)^{\alpha+1}}\right\rangle\\&+&\ \left\langle
T\wedge\beta^{p-1} , \chi'_k ((\log-u)^{-\alpha}){\alpha (\alpha +1)g^2
du\wedge
d^{\#}u \over u^2(\log-u)^{\alpha+2}}\right\rangle\\&+&\ \left\langle
T\wedge\beta^{p-1} , \chi''_k ((\log-u)^{-\alpha}){\alpha^2 g^2 du\wedge d^{\#}u
\over u^2(\log-u)^{2\alpha+2}}\right\rangle.
\end{array}
\end{equation}
Taking into account the positivity of the right hand terms in (\ref{e9}), and the fact that $\chi'_k\leqslant 1$, we deduce that
\begin{equation}\label{e11}
\left\langle T\wedge\beta^{p-1},g^2dd^{\#}\chi_k((\log-u)^{-\alpha})\right\rangle
\geqslant\left\langle T\wedge\beta^{p-1} , \chi'^2_k((\log-u)^{-\alpha}){{g^2du\wedge
d^{\#}u}\over u^2(\log-u)^{2\alpha+2}}\right\rangle.
\end{equation}
By (\ref{e8}) and (\ref{e10}), we get
$$\begin{array}{lcl}
\left\langle
T\wedge\beta^{p-1},g^2dd^{\#}\chi_k((\log-u)^{-\alpha})\right\rangle&\leqslant&
\left|I\right|+\left|II\right|+ {1\over 2}\left\langle T\wedge\beta^{p-1} ,
g^2\chi'^2_k ((\log-u)^{-\alpha}){{du\wedge d^{\#}u}\over
u^2(\log-u)^{2\alpha+2}}\right\rangle\\&+& 32\alpha^2\left\langle
T\wedge\beta^{p-1} , dg\wedge d^{\#} g\right\rangle.
\end{array}$$
Consequently, the claim follows by (\ref{e11}).
On the other hand the positivity of the terms involving (\ref{e9}) and the claim imply that there
exists a constant $C>0$ such that $$ \left\langle T\wedge\beta^{p-1} , \chi'_k
((\log-u)^{-\alpha}){g^2 du\wedge d^{\#}u \over u^2 (\log-u)^{1+\alpha}}
\right\rangle\leqslant \left\langle T\wedge\beta^{p-1}, g^2dd^{\#}\chi_k
((\log-u)^{-\alpha})\right\rangle\leqslant C .$$ We have in addition
$\chi'_k(t)=1$ if $t\geqslant\frac{3}{k}$,
then
$$\ds\int_{\{\Supp(g)\cap
\{(\log-u)^{-\alpha}\geqslant \frac{3}{k}\}\}\times\rb^n}g^2T\wedge\beta^{p-1}\wedge{{du\wedge d^{\#} u}\over
u^2(\log-u)^{1+\alpha}}\ \ \ {\rm {is\ bounded .}}$$ Therefore, if $k\rightarrow +\infty$,
then the integral
$\int_{\{\Supp(g)\smallsetminus K\}\times\rb^n}g^2T\wedge\beta^{p-1}\wedge{{du\wedge d^{\#} u}\over
u^2(\log-u)^{1+\alpha}}$ is finite.
For the case $T\wedge\beta^{p-1}$ is concave on $\{\Omega\smallsetminus K\}\times\rb^n$, we just observe that the first term $I$ in (\ref{e8}) is negative and therefore the proof of the claim is still holds.
\end{proof}
\begin{proofof}{\it Theorem \ref{t7}.} Let $v_k=\chi_{k^2}((\log(-u))^{-\frac{1}{k}})$, then $v_k$ increases
towards the characteristic function $\1_{\Omega\smallsetminus K}$. For every test function $\varphi$, it is clear that
\begin{equation}\label{e12}
\begin{array}{lcl}
\left\langle dd^{\#}(v_k T\wedge\beta^{p-1}),\varphi\right\rangle&=& \left\langle
v_k T\wedge\beta^{p-1},dd^{\#}\varphi\right\rangle\\&=&\left\langle
T\wedge\beta^{p-1},dd^{\#}(v_k\varphi)\right\rangle -
\left\langle T\wedge\beta^{p-1},\varphi dd^{\#} v_k \right\rangle\\&-&2
\left\langle T\wedge\beta^{p-1},dv_k\wedge d^{\#}\varphi\right\rangle.
\end{array}
\end{equation}
Thanks to the Cauchy-Schwarz inequality, we have
$$\begin{array}{lcl}\left|\left\langle T\wedge\beta^{p-1},dv_k\wedge
d^{\#}\varphi\right\rangle\right|&=&
\left|\left\langle T\wedge\beta^{p-1},
\chi'_{k^2}((\log-u)^{-\frac{1}{k}}){{\frac{1}{k}du\wedge d^{\#}\varphi}\over
(-u)(\log-u)^{{\frac{1}{k}}+1}}\right\rangle\right|\\&\leqslant&
\left\langle T\wedge\beta^{p-1}
,\left|\chi'_{k^2}((\log-u)^{-\frac{1}{k}})((\log-u)^{-\frac{1}{k}})\right|
{{\frac{1}{k^2}du\wedge d^{\#} u}\over u^2(\log-u)^{2}}\right\rangle^{\frac{1}{2}}
\\& &.\left\langle T\wedge\beta^{p-1}
,\left|\chi'_{k^2}((\log-u)^{-\frac{1}{k}})((\log-u)^{-\frac{1}{k}})\right|d\varphi\wedge d^{\#}\varphi
\right\rangle^{\frac{1}{2}}\\&\leqslant&
C\left\langle T\wedge\beta^{p-1},
{\1_{\Supp(\varphi)\smallsetminus K}}{{\frac{1}{k^2}du\wedge d^{\#} u}\over u^2(\log-u)^{2}}
\right\rangle^{\frac{1}{2}}.\left\langle T\wedge\beta^{p-1}
,d\varphi\wedge d^{\#}\varphi\right\rangle^{\frac{1}{2}}.
\end{array}$$
In view of Proposition \ref{6},  $\left\langle T\wedge\beta^{p-1},
{\1_{\Supp(\varphi)\smallsetminus K}}{{\frac{1}{k^2}du\wedge d^{\#} u}\over u^2(\log-u)^{2}}
\right\rangle$ goes to $0$ as $k\rightarrow+\infty$. Moreover, by hypothesis, $T$ has  locally finite mass
near $K$, then $\left\langle T\wedge\beta^{p-1}
,d\varphi\wedge d^{\#}\varphi\right\rangle$ is finite. So,
$$\lim_{k\rightarrow +\infty}\left|
\left\langle T\wedge\beta^{p-1},dv_k\wedge d^{\#}\varphi\right\rangle\right|=
\lim_{k\rightarrow +\infty}\left|
\left\langle T\wedge\beta^{p-1},d^{\#}v_k\wedge d\varphi\right\rangle\right|=0.$$
By virtue of (\ref{e12}), if $dd^{\#}T$ has locally finite mass near $K$, then
$$\left\langle \widetilde{dd^{\#}T\wedge\beta^{p-1}},\varphi\right\rangle-
\left\langle dd^{\#}\widetilde{T}\wedge\beta^{p-1},\varphi\right\rangle=
\displaystyle\lim_{k\rightarrow +\infty}
\left\langle T\wedge\beta^{p-1}\wedge dd^{\#} v_k,\varphi\right\rangle.$$
Then, the current $S=\widetilde{dd^{\#}T\wedge\beta^{p-1}}-dd^{\#}\widetilde{T}\wedge\beta^{p-1}=\ds\lim_{k\rightarrow +\infty}
T\wedge\beta^{p-1}\wedge dd^{\#} v_k$ is positive and supported by $K$. In the other case when $T\wedge\beta^{p-1}$ is concave in $\{\Omega\smallsetminus K\}\times\rb^n$, it suffices to show that $dd^{\#}T\wedge\beta^{p-1}$ is locally finite near $K$. For this aim, we rewrite (\ref{e12}) for $\varphi=g$ a positive test function which is equals $1$ in a neighborhood of ${\cal\rm O}\Subset\Omega$, we obtain
$$\begin{array}{lcl}
0\geqslant\ds\int_{\Omega\times\rb^n} v_k gdd^{\#}T\wedge\beta^{p-1}
&=&\ds\int_{\Omega\times\rb^n} v_k T\wedge\beta^{p-1}\wedge dd^{\#}g+\ds\int_{\Omega\times\rb^n} gT\wedge\beta^{p-1}\wedge dd^{\#}v_k
\\&+&2\left\langle T\wedge\beta^{p-1},dv_k\wedge d^{\#}g
\right\rangle\\&\geqslant&
\ds\int_{\Omega\times\rb^n} v_k T\wedge\beta^{p-1}\wedge
dd^{\#}g + 2\left\langle T\wedge\beta^{p-1},dv_k\wedge d^{\#}g
\right\rangle.
\end{array}$$
Hence, by the preceding proof, we see that the second and the third term in the right hand side of the last inequality goes to $0$, while the sequence $\int_{\Omega\times\rb^n} v_kT\wedge\beta^{p-1}\wedge dd^{\#}g$ is bounded because $T$ has locally finite mass near $K$. Thus, $dd^{\#}T$ has locally finite mass near $K$.
\end{proofof}
\begin{cor} Under the hypothesis of Theorem \ref{t7}, if we assume that $dT\wedge\beta^{p-1}$ has locally finite mass near $K$, then $d\widetilde{T}\wedge\beta^{p-1}=\widetilde{dT\wedge\beta^{p-1}}$.
\end{cor}
In particular, we recover a very recently result obtained by Berndtsson \cite{11}, saying that if $T$ is minimal then $\widetilde T$ is also minimal. The following example shows that the trivial extension $\widetilde{dT}$ (in particularly $\widetilde{dT\wedge\beta^{p-1}}$) does not always exist in general even in the case where $T$ is smooth and $dd^{\#}T=0$ on $\{\Omega\smallsetminus K\}\times\rb^n$.
\begin{exe} Let, $$T=\left(1-\sin\frac{1}{(x_1+x_2)^2}\right)dx_1\wedge d\xi_1+\left(1+\sin\frac{1}{(x_1+x_2)^2}\right)dx_2\wedge d\xi_2,\ \forall x_1,x_2\in\rb.$$ The current $T$ is weakly positive of bidimension $(1,1)$ on $\{\rb^2\smallsetminus\{x_1+x_2=0\}\}\times\rb^2$, and it is clear that $\widetilde{T}$ and $\widetilde{dd^{\#}T}$ are exist, since $T$ is smooth and $dd^{\#}T=0$. Moreover, it is obvious that the coefficients of $dT$ are $\pm\frac{2}{(x_1+x_2)^3}\cos\frac{1}{(x_1+x_2)^2}$. Therefore, a simple computation yields
$$\lim_{n\rightarrow+\infty}\int_{\lbrace|x_2|<1,\frac{1}{\sqrt{2n\pi+\frac{\pi}{4}}}<|x_1+x_2|<\frac{1}{\sqrt{2n\pi}}\rbrace\times\rb^2}\frac{2}{|x_1+x_2|^3}\left|\cos\frac{1}{(x_1+x_2)^2}\right|dx_1\wedge d\xi_1\wedge dx_2\wedge d\xi_2=+\infty.$$
It follows that $dT$ has an infinite mass near $\{x_1+x_2=0\}$.
\end{exe}
\begin{proof} Let $\rho:\rb\l\rb$ be a smooth positive function such that $\rho(t)=0$ if $t<\frac{1}{2}$ and  $\rho(t)=1,$ for $t>1$. Let's denote by $\rho_r (t)=\rho(\frac{t}{r})$ then $\widetilde T=\displaystyle\lim_{r\rightarrow 0} \rho_r
((\log-u)^{-\alpha})T$.  Let $\varphi\in{\mathscr D}^{0,1}(\mathbb{R}^n\times\mathbb{R}^n)$, then we have
$$\begin{array}{lcl}
\left\langle d\widetilde T\wedge\beta^{p-1},\varphi \right\rangle &=& -\left\langle
\widetilde T\wedge\beta^{p-1}, d\varphi\right\rangle\\&=& -\displaystyle\lim_{r\rightarrow
0}\left\langle \rho_r ((\log-u)^{-\alpha})T\wedge\beta^{p-1},d\varphi\right\rangle\\&=&
-\displaystyle\lim_{r\rightarrow 0}\left\langle
T\wedge\beta^{p-1},d\left(\rho_r((\log-u)^{-\alpha})\varphi\right)\right\rangle +\displaystyle\lim_{r\rightarrow
0}\left\langle T\wedge\beta^{p-1},d\rho_r((\log-u)^{-\alpha})\wedge\varphi\right\rangle\\&=&
\left\langle
\widetilde{dT\wedge\beta^{p-1}},\varphi\right\rangle +\displaystyle\lim_{r\rightarrow
0}\left\langle T\wedge\beta^{p-1},d\rho_r((\log-u)^{-\alpha})\wedge\varphi\right\rangle.
\end{array}$$
By the Cauchy-Schwarz inequality, we get
$$\begin{array}{lcl}
 \left|\left\langle T\wedge\beta^{p-1},d\rho_r((\log-u)^{-\alpha})\wedge\varphi\right\rangle
\right|&=& \left|\left\langle T\wedge\beta^{p-1} ,{1\over r} \rho'(r^{-1}(\log-u)^{-\alpha}){{\alpha
du\wedge\varphi}\over (-u)(\log-u)^{\alpha+1}}\right\rangle\right|\\&\leqslant&
\left\langle T\wedge\beta^{p-1}
,\left|\rho'(r^{-1}(\log-u)^{-\alpha})r^{-1}(\log-u)^{-\alpha}\right|
{{\alpha^2 du\wedge d^{\#} u}\over u^2(\log-u)^{2}}\right\rangle^{\frac{1}{2}}\\&
&.\left\langle T\wedge\beta^{p-1}
,\left|\rho'(r^{-1}(\log-u)^{-\alpha})r^{-1}(\log-u)^{-\alpha}\right| J(\varphi)\wedge
\varphi\right\rangle^{\frac{1}{2}}.
\end{array}$$
Setting $C=\rm{sup}\{|t\rho'(t)|,\ t\in\rb
\}$.
Then
$$\begin{array}{lcl}
\left|\left\langle T\wedge\beta^{p-1},d\rho_r((\log-u)^{-\alpha})\wedge\varphi\right\rangle
\right|&\leqslant&
C^{\frac{1}{2}}\left\langle T\wedge\beta^{p-1},
{\1_{\Supp(\varphi)\smallsetminus K}}{{\alpha^2 du\wedge d^{\#} u}\over u^2(\log-u)^{2}}
\right\rangle^{\frac{1}{2}}\\& &.\left\langle T\wedge\beta^{p-1}
,\left|\rho'(r^{-1}(\log-u)^{-\alpha})r^{-1}(\log-u)^{-\alpha}\right| J(\varphi)\wedge
\varphi\right\rangle^{\frac{1}{2}}.
\end{array}$$
Thanks to Proposition \ref{6}, the first term in the right hand side is finite.
By the Lebesgue's theorem, we have
$$\lim_{r\rightarrow
0}\left\langle T\wedge\beta^{p-1}
,\left|\rho'(r^{-1}(\log-u)^{-\frac{1}{2}})r^{-1}(\log-u)^{-\frac{1}{2}}\right|J(\varphi)\wedge
\varphi\right\rangle=0 .$$
Hence, we conclude that $\ds\lim_{r\rightarrow
0}\left|\left\langle T\wedge\beta^{p-1},d\rho_r((\log-u)^{-\frac{1}{2}})\wedge\varphi\right\rangle\right|=0$.
\end{proof}
\begin{cor} Let $\Omega$ be an open subset of $\rb^n$ and  $T$ be a weakly positive current of bidimension $(p,p)$ on $\Omega\times\rb^n$ such that $T\wedge\beta^{p-1}$ is convex on $\Omega\times\rb^n$. Then, for every compact subset $K$ of $\Omega$ with sigma-finite $(p-2)$-dimensional Hausdorff measure, the current $\1_KT$ is weakly positive and the current $\1_KT\wedge\beta^{p-1}$ is convex on $\Omega\times\rb^n$.
\end{cor}
\begin{proof} Let $T'=\1_{\Omega\smallsetminus K}T$, then $\widetilde{T'}=T-\1_KT$. Thanks to Theorem \ref{t7}, we obtain
$$\widetilde{dd^{\#}T'\wedge\beta^{p-1}}=dd^{\#}\widetilde{T'}\wedge\beta^{p-1}+S,$$
where $S$ is a positive measure. Therefore,
$$\begin{array}{lcl}
dd^{\#}(\1_KT)\wedge\beta^{p-1}=dd^{\#}(T-\widetilde{T'})\wedge\beta^{p-1}&=&dd^{\#}T\wedge\beta^{p-1}-dd^{\#}\widetilde{T'}\wedge\beta^{p-1}\\&=&dd^{\#}T\wedge\beta^{p-1}-\widetilde{dd^{\#}T'\wedge\beta^{p-1}}+S\\&=&(dd^{\#}T\wedge\beta^{p-1})-\widetilde{\1_{\Omega\smallsetminus K}(dd^{\#}T\wedge\beta^{p-1})}+S\\&=&\1_{K}(dd^{\#}T\wedge\beta^{p-1})+S,
\end{array}$$
and the desired result will follows, since $T\wedge\beta^{p-1}$ is convex.
\end{proof}

\end{document}